\documentclass[12pt]{amsart}%
\usepackage[T1]{fontenc}
\usepackage[english]{babel}
\usepackage[utf8]{inputenc}
\usepackage{indentfirst}
\usepackage{amsmath, soul, enumerate}
\usepackage{amsthm}
\usepackage{amsfonts}
\usepackage{amssymb}

\usepackage{amsmath,bm}

\usepackage{graphicx}

\usepackage[colorlinks = true,
            linkcolor = black,
            urlcolor  = black,
            citecolor = black,
            anchorcolor = blacke]{hyperref}

\usepackage{enumitem, soul}

\usepackage{pict2e}

\usepackage{mathtools}

\DeclareMathOperator{\DM}{DM}

\DeclareMathOperator{\BDM}{\mathbf{DM}}

\DeclareMathOperator{\Max}{Max}

\DeclareMathOperator{\Min}{Min}

\newtheorem{theorem}{Theorem}[section]

\newtheorem{definition}[theorem]{Definition}

\newtheorem{lemma}[theorem]{Lemma}

\newtheorem{remark}[theorem]{Remark}

\newtheorem{example}[theorem]{Example}

\newtheorem{corollary}[theorem]{Corollary}

\title{Tense logics based on posets}

\author[Chajda, L\"anger, Ledda, Paseka, Vergottini]{Ivan Chajda$^{1,5}$, Helmut L\"anger$^{2,6}$, \\ 
	Antonio Ledda$^{3,7}$, Jan Paseka$^{4,6}$ and Gandolfo Vergottini$^3$}

\date{}

\begin{document}

\footnotetext[1]{Faculty of Science, Department of Algebra and Geometry, Palack\'y University Olomouc, 17.\ listopadu 12, 771 46 Olomouc, Czech Republic}

\footnotetext[2]{Faculty of Mathematics and Geoinformation, Institute of Discrete Mathematics and Geometry, TU Wien, Wiedner Hauptstra\ss e 8-10, 1040 Vienna, Austria, and Faculty of Science, Department of Algebra and Geometry, Palack\'y University Olomouc, 17.\ listopadu 12, 771 46 Olomouc, Czech Republic}

\footnotetext[3]{A.Lo.P.Hi.S Research Group, %
University of Cagliari, Via Is Mirrionis 1, 09123, Cagliari, Italy}

\footnotetext[4]{Faculty of Science, Department of Mathematics and Statistics, Masaryk University Brno, Kotl\'a\v rsk\'a 2,
611 37 Brno, Czech Republic}

\footnotetext[5]{Support of the research of the Czech Science Foundation (GAČR) project 20-09869L ``The many facets of orthomodularity'' and by IGA, project P\v rF~2024~011, is gratefully acknowledged.}

\footnotetext[6]{The authors acknowledge the support by the bilateral Austrian Science Fund (FWF) project I 4579-N and Czech Science Foundation (GAČR) project 20-09869L ``The many facets of orthomodularity''.}

\footnotetext[7]{Antonio Ledda (corresponding author), antonio.ledda@unica.it}


\maketitle

\begin{abstract}
Not all logical systems can be captured using algebras. We see this in classical logic (formalized by Boolean algebras) and many-valued logics (like \L ukasiewicz logic with MV-algebras). Even quantum mechanics, initially formalized with orthomodular lattices, benefits from a simpler approach using just partially ordered sets (posets).

This paper explores how logical connectives are introduced in poset-based logics. Building on prior work by the authors, we delve deeper into ``dynamic'' logics where truth values can change over time. We consider time sets with a preference relation and propositions whose truth depends on time.

Tense operators, introduced by J.Burgess and extended for various logics, become a valuable tool.  This paper proposes several approaches to this topic, aiming to inspire a further stream of research.

\end{abstract}

{\bf AMS Subject Classification:} 03B44, 03G12, 03G25, 06A11, 08A55

{\bf Keywords:} Poset, orthomodular poset, Dedekind-MacNeille completion,  implication, adjoint operator, adjoint pair, tense operator, tense logic, time frame

\section{Introduction}
A great deal of logics of definite importance to contemporary mathematics are semilattice based \cite{F}.
Precisely, let $\mathtt K$ be a class of algebras of the same arbitrary similarity type. The class $\mathtt K$ is said to be semilattice-based when each algebra in $\mathtt K$ has a semilattice reduct uniformly defined by some primitive or term-defined operation. Cases in point may vary from classical logic to intuitionistic and many valued logics, whose algebraic counterparts are of great importance to algebraic logic, and all these algebraic structures (Boolean algebras, Heyting algebras, MV algebras, MTL algebras, residuated lattices etc. \cite{GH, CDM, Ha, GJKO}) are based on lattice reducts.
However, there are also other logics based on posets only.
At the beginning of the twentieth century it was recognized that the logic of quantum mechanics differs essentially from classical propositional calculus based algebraically on Boolean algebras. Husimi \cite{Hu} and Birkhoff together with von Neumann \cite{BVN} introduced orthomodular lattices in order to serve as an algebraic base for the logic of quantum mechanics, see \cite{Ber}. These lattices incorporate many aspects of this logic with one exception.
In fact, within the logic of quantum mechanics the disjunction of two propositions need not exist in the case when these propositions are neither orthogonal nor comparable. This fact motivated a number of researchers to consider orthomodular posets instead of orthomodular lattices within their corresponding investigations, see e.g. \cite{PP} and references therein.

However, such a logic is not dynamic in the sense that it does not incorporate time dimension.

Perhaps, a few words on this statement may be useful. A propositional logic, either classical or non-classical, usually does not incorporate the dimension of time. In order to organize this logic as a ``tense logic'' (or ``time logic'' according to another terminology, see \cite{FP, FGV, RU, VB, CPZ, He}] we usually construct the tense operators P , F , H and G. Their meaning is:
\begin{itemize}
\item[P:] “It has at some time been the case that”;
\item[F:] “It will at some time be the case that”;
\item[H:] “It has always been the case that”;
\item[G:] “It will always be the case that”.
\end{itemize}

For finite orthomodular posets, a tense logic has been discussed recently in \cite{CL1}.
Following ideas from \cite{Bu, DG}, in \cite{CL1} two of the present authors consider first a time frame (or time scale) $(T, R)$, with the set of times $T$ nonempty and $R$ a serial binary relation (eventually reflexive) to intuitively meant that the time flow never begins and never ends.

Following the ideas in \cite{CL1}, in this paper we will generalize the approach to study logics based on posets satisfying rather minimal condition called $MLUB$-completness.
Natural examples of such posets are, e.g., finite posets, posets 
 with no infinite ascending or descending chains, or complete lattices.

The paper is structured as follows: in section \ref{sc:prl} we provide some basic notions. In section \ref{sectiontense}, we introduce the concept of tense operators in a poset ${\mathbf A}$ that satisfies quite general conditions (page \pageref{def:tnspr}) and we discuss their properties, and after analyzing some critical aspects we consider the cases when they form either a dynamic pair or a Galois connection.

Section  \ref{sec:tnsdmc} focuses on establishing a connection between the familiar tense operators \(P\), \(F\), \(H\), and \(G\) (introduced earlier in Section~\ref{sectiontense}) and their counterparts  $\widehat{G}, \widehat{P}, \widehat{H}$ 
 and $\widehat{F}$  in the Dedekind-MacNeille completion ${\DM({\mathbf A})}$ of the ${\mathbf A}$.

As shown in section \ref{sectiontense}, given a time frame \((T, R)\), in section \ref{sec:cnstrctns}, we define the tense operators \(P\), \(F\), \(H\), and \(G\) on the logic derived from a bounded poset ${\mathbf A}$, thereby obtaining a dynamic poset. 
However, the question arises whether, conversely, given a dynamic poset, one can find a time preference relation such that the induced tense operators coincide with the given ones. Theorem \ref{th4}
shows that this is indeed possible: from the given tense operators, we can construct a binary relation \(R^*\) on \(T\) such that the tense operators induced by the time frame \((T, R^*)\) are comparable to the given tense operators with respect to the quasi-order relations \(\leq_1\) and \(\leq_2\). Furthermore, if \(P\), \(F\), \(H\), and \(G\) are induced by some time frame, then the newly constructed tense operators are equivalent to the given ones via the equivalence relations \(\approx_1\) and \(\approx_2\) as described in Definition \ref{dfn:rlzord}.

In section \ref{sec:mixed}, we explore the introduction of the ``unsharp'' or ``inexact'' logical connectives, namely conjunction $\odot$ and implication $\to$, within the context of a poset $(A, \leq)$ that satisfies both the Antichain Condition (ACC) and the Descending Chain Condition (DCC). Our aim is to demonstrate how these connectives form a sort of adjoint pair. Furthermore, we extend the definition of these connectives to subsets of a poset, with the objective of establishing a relationship with tense operators. We conclude the section by presenting several results regarding this relationship.

Finally in section \ref{sec:cnnctvs}, we focus our attention to the case in which the connectives are already defined in a way that the structure forms a residuated poset, and then investigate their relationship with the tense operators.

We will begin by providing the formal notion of a residuated poset and then demonstrate some of its basic properties. Subsequently, we will present a method to combine the application of logical connectives with tense operators. Indeed, in this case, we encounter the problem of needing to generalize operations between elements to operations between subsets.

\section{Preliminaries}\label{sc:prl}

In this section, we introduce some useful basic concepts related to partially ordered sets (posets). 

A poset is a set equipped with a partial order. Formally, a poset $\mathbf{A} = (A, \leq)$ consists of a set $A$ and a binary relation $\leq$ on $A$ that is reflexive, antisymmetric, and transitive.

Consider a poset $\mathbf{A} = (A, \leq)$ and let $X, Y \subseteq A$. If $\mathbf{A}$ has a bottom element, this element will be denoted by $0$. If $\mathbf{A}$ has a top element, this element will be denoted by $1$. The poset $\mathbf{A}$ is called \emph{bounded} if it has both a bottom element and a top element. In this case, it will be denoted by $\mathbf{A} = (A, \leq, 0, 1)$.
We proceed assuming all posets are bounded.



In what follows, we will need to use \emph{binary operators} on $A$. These operators are simply maps from $A^2$ to $2^A$, i.e., they assign to each pair $(x, y)$ of elements of $A$ a subset of  $A$. For the sake of brevity, we will not distinguish between a singleton $\{ a\}$ and its unique element $a$.


To address a specific scenario within the context of posets, we need to recall some definitions.

\begin{definition}
A poset $\mathbf A$ is said to satisfy the following conditions:
\begin{itemize}
	
	\item the \emph{Ascending Chain Condition} (shortly \emph{ACC}), if there are no infinite ascending chains in $\mathbf A$;
	
	\item the \emph{Descending Chain Condition} (shortly \emph{DCC}), if there are no infinite descending chains in $\mathbf A$.
	\end{itemize}
\end{definition}

Let us notice that every finite poset satisfies both the ACC and the DCC. Furthermore, let $\Max A$ and $\Min A$ denote the set of all maximal and minimal elements of $\mathbf{A}$, respectively. If $\mathbf{A}$ satisfies the ACC, then for every $a\in A$, there exists some $b\in\Max A$ with $a\leq b$. This implies that if $A$ is not empty, the same is true for $\Max A$. The corresponding assertion holds for the DCC and $\Min A$.
We can summarize this in the following definition:

\begin{definition}
	Let $\mathbf A$ be a poset. We say that $\mathbf A$ is:

\begin{enumerate}
	\item \emph{MUB-complete} if for every upper bound $x$ of a non-empty subset $M$ of $A$, there is a minimal upper bound of $M$ below $x$;
	
	\item \emph{MLB-complete} if for every lower bound $x$ of a non-empty subset $M$ of $A$, there is a maximal lower bound of $M$ above $x$;
	
	\item \emph{MLUB-complete} if it is both MUB-complete and MLB-complete. 
\end{enumerate}	
\end{definition}

Evidently, if $\mathbf{A}$ satisfies the ACC, then $\mathbf{A}$ is MLB-complete, and the dual statement holds for the DCC. Moreover, every finite poset and every complete lattice are MLUB-complete.

We will need the following two useful lemmas.

\begin{lemma}\label{minmaxtMLUB}
Let $\mathbf A$ be an MLUB-complete poset and $M$ a non-empty subset of $A$. 
Then 
\begin{enumerate}[label={\rm{(\roman*)}}]
    \item $U(\Max L(M))=UL(M)$;
    \item $L(\Min U(M))=LU(M)$.
\end{enumerate}
\end{lemma}
\begin{proof} (i): Evidently, $\Max L(M)\subseteq L(M)$. We conclude 
 $U(\Max L(M))\supseteq UL(M)$. 

 Conversely, let $x\in U(\Max L(M))$ and $z\in L(M)$. Then $x\geq y$ for all 
 $y\in \Max L(M)$. Since $\mathbf A$ is  MLUB-complete we have that 
 $z\leq y_0$ for some $y_0\in \Max L(M)$. Therefore $x\geq z$, i.e., 
 $x\in UL(M)$.

 Part (ii) can be proven similarly. 
\end{proof}

\begin{lemma}\label{productMLUB} Let $({\mathbf A}_{i})_{i\in I}$ be 
a family of MLUB-complete posets. Then the cartesian product 
  $\prod_{i\in I}{\mathbf A}_{i}$  is an MLUB-complete poset. 
  Moreover, for every  non-empty subset of  $\prod_{i\in I}{\mathbf A}_{i}$  
  we have 
  \begin{enumerate}[label={\rm{(\roman*)}}]
  \item $\Min U_{\prod_{i\in I}{\mathbf A}_{i}}(M)=%
    \bigtimes_{i\in I} \Min U_{{\mathbf A}_{i}}(\{m(i)\mid m\in M\})$;
    \item $\Max L_{\prod_{i\in I}{\mathbf A}_{i}}(M)=%
    \bigtimes_{i\in I} \Max L_{{\mathbf A}_{i}}(\{m(i)\mid m\in M\})$.
\end{enumerate}
\end{lemma}
\begin{proof}
    (i) Let $M$ be  a non-empty subset of  $\prod_{i\in I}{\mathbf A}_{i}$ 
    and $x$ its upper bound. Let $i\in I$ and $M_i=\{m(i)\mid m\in M\}$. 
    Then $x(i)\in U(M_i)$ and there is 
    a minimal upper bound $q_i$ of $M_i$ below $x(i)$. 
    Let $q$ be an element of $\prod_{i\in I}{\mathbf A}_{i}$ defined by 
    $q(i)=q_i$ for every $i\in I$. Evidently, $q$  is 
     a minimal upper bound of $M$ below $x$. Hence $({\mathbf A}_{i})_{i\in I}$ 
     is MUB-complete. 
     
     That $({\mathbf A}_{i})_{i\in I}$ 
     is MLB-complete and (ii) holds works the same way for similar reasons.
\end{proof}

For any set $A$, we denote the set of its non-empty subsets by $\mathcal{P}_+(A)$. If $f: A \to A'$ is a function and $T$ is a set, we define $\mathcal{P}_{+}f^{T}$ as a function from $\mathcal{P}_+(A^{T})$ to $\mathcal{P}_+((A')^{T})$ by the formula 
\[  \big(\mathcal P_+f^{T}\big)(B)=\{(f(b(t))_{t\in T}\mid b=(b(t))_{t\in T}\in B)\}\]
 for all $B\in\mathcal P_+(A^T)$.

Let us introduce different kinds of relations that will be useful during the development of our objectives.
\begin{definition}\label{dfn:rlzord}
	Let $\mathcal A$ be a poset. For  $X,Y \subseteq \mathcal P_+A$ we define the following six preorders $\leq$, $\leq_1$, $\leq_2$, $\sqsubseteq$, $\approx_1$,  and $\approx_2$ as follows:
\begin{enumerate}
	\item $X\leq Y$   if for every $x\in X$ and  for every $y\in Y$ it holds $x\leq y$;
	\item  $X\leq_1 Y$   if for every $x\in X$ there exists some $y\in Y$ with $x\leq y$; 
	\item  $X\leq_2 Y$   if for every $y\in Y$ there exists some $x\in X$ with $x\leq y$; 
	\item  $X\sqsubseteq Y$  if there exists some $x\in X$ and some $y\in Y$ with $x\leq y$; 
	\item  $X\approx_1 Y$  if both $X\leq_1 Y$ and $Y\leq_1 X$;
	\item  $X\approx_2 Y$  if both $X\leq_2 Y$  and $Y\leq_2 X$.
\end{enumerate}	
\end{definition}

Moreover, we define the \emph{lower cone} and \emph{upper cone} of $X$, as the following sets, respectively:

\begin{enumerate}
\item	$L(X) := \{a\in A\mid a\leq X\}$, 
\item	$U(X):= \{a\in A\mid X\leq a\}$.
\end{enumerate}
where  $X \subseteq A$ and $a \in A$.

Let us conclude this preliminary section with the statement, without proof, of Lemma \ref{usefulprop} and Corollary \ref{minmaxcor}, which list several useful and obvious properties of $\leq$, $\leq_1$, $\leq_2$, $\sqsubseteq$, $\approx_1$, and $\approx_2$.

\begin{lemma}\label{usefulprop}

Let $\mathbf A=(A,\leq)$  be	a poset and let $X,Y\subseteq A$. Then 

	\begin{enumerate}[label={\rm{(\roman*)}}]

	\item if $X\subseteq Y$ then $X\leq_1 Y$ and $Y\leq_2 X$, 

	\item if $X\leq_1 Y$ then $U(Y)\subseteq U(X)$ and $LU(X)\subseteq LU(Y)$,  

		\item if $X\leq_2 Y$ then $L(X)\subseteq L(Y)$ and  $UL(Y)\subseteq UL(X)$, 

		\item if  $X,Y$ are antichains and $ X\approx_1 Y $ or $ X\approx_2 Y $ then $X=Y$. 

\end{enumerate}

\end{lemma}

\begin{corollary}\label{minmaxcor}

		Let $\mathbf A=(A,\leq)$ be an MLUB-complete poset. Then 

			\begin{enumerate}[label={\rm{(\roman*)}}]

			\item if $X\leq_1 Y$ then $ \Min U(X)\leq_2 \Min U(Y)$, 

			\item if $LU(X)\subseteq LU(Y)$, $X=\Max LU(X)$ and $Y=\Max LU(Y)$ then $X\leq_1 Y$,

			\item if $X\leq_2 Y$ then $ \Max L(X)\leq_1 \Max L(Y)$, 

			\item if $UL(Y)\subseteq UL(X)$, $Y=\Min UL(Y)$ and $X=\Min UL(X)$ then $X\leq_2 Y$.

		\end{enumerate}

\end{corollary}

Finally, consider a poset  $(A,\leq)$. If $a, b\in A$ and $sup(a,b)$ exists then we will denote it by $a\vee b$. If  $inf(a,b)$ exists then we will denote it by $a\wedge b$. 
An {\em antitone involution} on a poset $(A,\leq)$ is a mapping $'$ from $A$ to $A$ satisfying the following conditions for all $x,y\in A$:

\begin{enumerate}[label={\rm{(\roman*)}}]

	\item $x\leq y$ implies $y'\leq x'$,

	\item $x''=x$.

\end{enumerate}
A {\em complementation} on a bounded poset $(A,\leq,0,1)$ is a mapping $'$ from $A$ to $A$ satisfying $x\vee x'=1$ and $x\wedge x'=0$ for all $x\in A$.

\section{Tense operators}\label{sectiontense}
Generally, the expressive power of propositional logics does not include the ability to distinguish facts with respect to ``before'' and ``after''. To incorporate this capability, tense operators have been introduced. Tense operators allow us to express temporal relationships within a logical framework. These operators expand the language of the logic by enabling statements about the past, present, and future. The initial work on tense operators was conducted within the context of classical logic, integrating them into the framework of Boolean algebras \cite{Bu, Jo}. For an introduction to the theory of tense operators in algebraic structures, see the monograph  \cite{CP15a}.

Let us recall the basic tense operators $P$, $F$, $H$, $G$ commonly used, along with their meanings:
\begin{align*}
P & \ldots\text{``It has at some time been the case that''}, \\
F & \ldots\text{``It will at some time be the case that''}, \\
H & \ldots\text{``It has always been the case that''}, \\
G & \ldots\text{``It will always be the case that''}.
\end{align*}

Let us consider a poset \((A, \leq)\) along with a non-empty time set \(T\). We interpret the elements \(p\) of \(A^T\) as time-dependent events. That is, for \(t \in T\), the symbol \(p(t)\) represents the value of the event \(p\) at time \(t\). It is evident that the operators \(P\) and \(F\) act as existential quantifiers over the past and future segments of \(T\), respectively, while \(H\) and \(G\) serve as universal quantifiers over the corresponding segments. Consequently, for each \(t \in T\) and every \(p \in A^T\), we have:
\begin{align*}
p(t)\leq P(p)(t) & \text{ and }p(t)\leq F(p)(t), \\
H(p)(t)\leq p(t) & \text{ and }G(p)(t)\leq p(t),
\end{align*}

shortly
\[ H(p)\leq p\leq P(p)\text{ and }G(p)\leq p\leq F(p).\]

To distinguish between the ``past'' and ``future'' with respect to the time set \(T\), we introduce a time-preference relation, i.e. a non-empty binary relation \(R\) on \(T\).
\begin{definition}
	A \emph{time frame} is a pair \(\mathbf{T} = (T, R)\), where \(T\) is a non-empty set called the time set, and \(R\) is a binary relation on \(T\). For \(s, t \in T\) with \(s \mathrel{R} t\), we say that 
\begin{center}
``$s$ is before $t$'' or ``$t$ is after $s$''. 
\end{center}

\end{definition}

The relation \( R \) in a time frame may satisfy specific properties to accurately model a particular concept of time. Thus, these properties can vary depending on the context and requirements of the specific temporal model. For example, in some temporal logics, it may be useful to consider only transitive and irreflexive relations, while in others, it may be useful to include density or linearity of the relation. 

For our purposes in this paper, we will consider only ``serial relations'' (see \cite{CP15a}), i.e., relations \( R \) such that for each \( s \in T \) there are some \( r, t \in T \) with \( r \mathrel{R} s \) and \( s \mathrel{R} t \). Clearly, if \( R \) is reflexive, then it is serial. Usually, \( R \) is considered to be a partial order relation or a quasi-order \cite{Bu, DG, FP}. Moreover, we will denote by \( R^{-1} \) the ``inverse of'' \( R \), i.e., \( s \mathrel{R^{-1}} t \) if and only if \( t \mathrel{R} s \). 

We put \( \mathbf{T}^{op} = (T, R^{-1}) \). Recall that if \( \mathbf{T} \) is serial, then also \( \mathbf{T}^{op} \) is serial.
If \( \mathbf{A} = (A, \leq) \) is a poset, we define \( \geq = (\leq)^{-1} \). Then \( \mathbf{A}^{op} = (A, \geq) \) is also a poset and is called the ``dual of'' \( \mathbf{A} \). Clearly, if \( \mathbf{A} = (A, \leq) \) is a poset satisfying both the ACC and the DCC, then also \( \mathbf{A}^{op} \) is a poset satisfying both the ACC and the DCC, and \( \Min X \) in \( \mathbf{A} \) coincides with \( \Max X \) in \( \mathbf{A}^{op} \) for every subset \( X \) of \( A \). Similarly, if \( \mathbf{A} = (A, \leq) \) is an MLUB-complete poset, then \( \mathbf{A}^{op} \) is an MLUB-complete poset.

\begin{definition}\label{def:tnspr}
	Let \(\mathbf{A} = (A, \leq)\) be an MLUB-complete poset,  $\mathbf T=(T,R)$ a time frame, and consider the set \(A^T\) of all time-dependent propositions on \(\mathbf{A}\). We define the tense operators \( {}_{\mathbf{T}}P_{\mathbf{A}} \), \( {}_{\mathbf{T}}F_{\mathbf{A}} \), \( {}_{\mathbf{T}}H_{\mathbf{A}} \), and \( {}_{\mathbf{T}}G_{\mathbf{A}} \) on \(\mathbf{A}\) to be the following mappings from \(\mathcal{P}_+(A^T)\) to \((\mathcal{P}_+A)^T\):
\begin{align*}
	{}_{\mathbf T}P_{\mathbf A}(B)(s) & :=\Min U\big(\{q(t)\mid q\in B\text{ and }t\mathrel Rs\}\big), \\
	{}_{\mathbf T}F_{\mathbf A}(B)(s) & :=%
	\Min U\big(\{q(t)\mid q\in B\text{ and }s\mathrel Rt\}\big), \tag{To}\\
	{}_{\mathbf T}H_{\mathbf A}(B)(s) & :=\Max L\big(\{q(t)\mid q\in B\text{ and }t\mathrel Rs\}\big), \\
	{}_{\mathbf T}G_{\mathbf A}(B)(s) & :=\Max L\big(\{q(t)\mid q\in B\text{ and }s\mathrel Rt\}\big),
\end{align*}

for all $B\in\mathcal P_+(A^T)$ and all $s\in T$. 
\end{definition}

The operators defined in this manner are said to be induced by \((T,R)\). If there is no misunderstanding, we will simply write \(P\), \(F\), \(H\), and \(G\) instead of \({}_{\mathbf{T}}P_{\mathbf{A}}\), \({}_{\mathbf{T}}F_{\mathbf{A}}\), \({}_{\mathbf{T}}H_{\mathbf{A}}\), and \({}_{\mathbf{T}}G_{\mathbf{A}}\), respectively.

\begin{lemma} \label{xproptense}

	Let $\mathbf A=(A,\leq)$ be an MLUB-complete poset, $\mathbf T=(T,R)$ a time frame. Then, for all $B\in\mathcal P_+(A^T)$, 
 \begin{enumerate}[label={\rm{(\roman*)}}]
     \item ${}_{\mathbf T}P_{\mathbf A}(L(B))=%
     {}_{\mathbf T}P_{\mathbf A}\Max (L(B))$, 
     \item ${}_{\mathbf T}F_{\mathbf A}(L(B))=%
     {}_{\mathbf T}F_{\mathbf A}\Max (L(B))$, 
     \item ${}_{\mathbf T}H_{\mathbf A}(U(B))=%
     {}_{\mathbf T}H_{\mathbf A}\Min (U(B))$, 
     \item ${}_{\mathbf T}G_{\mathbf A}(U(B))=%
     {}_{\mathbf T}G_{\mathbf A}\Min (U(B))$. 
 \end{enumerate}
\end{lemma}

\begin{proof} (i): Let $B\in\mathcal P_+(A^T)$ and  $s\in T$. 
Since $\Max (L(B))\subseteq L(B)$ we immediately conclude that 
$U\big(\{q(t)\mid q\in L(B)\text{ and }t\mathrel Rs\}\big)\subseteq %
U\big(\{q(t)\mid q\in \Max (L(B))\text{ and }t\mathrel Rs\}\big)$.
Conversely, let $r\in A$ and $r\in U\big(\{q(t)\mid q\in \Max (L(B))\text{ and }t\mathrel Rs\}\big)$. Then $r\geq q(t)$ for all $q\in \Max (L(B))$ and 
$t\in T$ such that $t\mathrel Rs$. Let $p\in L(B)$ and $t_0\mathrel Rs$, 
$t_0\in T$. Then $p\leq q_0$ for some $q_0\in \Max (L(B))$. 
We conclude that $p(t_0)\leq q(t_0)\leq r$, i.e., 
$r\in U\big(\{q(t)\mid q\in L(B)\text{ and }t\mathrel Rs\}\big)$.
Hence ${}_{\mathbf T}P_{\mathbf A}(L(B))=%
     {}_{\mathbf T}P_{\mathbf A}\Max (L(B))$.

Parts (ii) - (iv) can be proven similarly.
\end{proof}

The definition of \(P\), \(F\), \(H\), and \(G\) reveals a duality among them. Consequently, we formulate and prove some results only for one of them. The corresponding results for the remaining operators are dual.

\begin{theorem}\label{propdual}
Let $\mathbf A=(A,\leq)$ be an MLUB-complete poset, $\mathbf T=(T,R)$ a time frame. Then, 
\[  {}_{\mathbf T}P_{\mathbf A}={}_{\mathbf T}H_{{\mathbf A}^{op}}%
={}_{{\mathbf T}^{op}}F_{{\mathbf A}}={}_{{\mathbf T}^{op}}G_{{\mathbf A}^{op}}. \]

Moreover, if $\mathbf A'=(A',\leq')$ is a poset and $f\colon \mathbf A \to \mathbf A'$ 	is an isomorphism of posets, then 

\[(\mathcal P_+f)^T \circ {}_{\mathbf T}X_{\mathbf A}={}_{\mathbf T}X_{\mathbf A'}\circ \mathcal P_+(f^T),\]

for all $X\in\{P,F,H,G\}$.
\end{theorem}

\begin{proof} Let $B\in\mathcal P_+(A^T)$, $a\in A$ and  $s\in T$. We compute:
		\begin{align*}
	a\in {}_{\mathbf T}P_{\mathbf A}(B)(s)&\text{ if and only if } %
			a\in \Min_{\mathbf A} U_{\mathbf A}\big(\{q(t)\mid q\in B\text{ and }t\mathrel Rs\}\big)\\
	&\text{ if and only if } %
		a\in \Max_{{\mathbf A}^{op}} L_{{\mathbf A}^{op}}\big(\{q(t)\mid q\in B%
		\text{ and }t\mathrel Rs\}\big)\\
		&\text{ if and only if } %
		a\in {}_{\mathbf T}H_{{\mathbf A}^{op}}(B)(s)\\
		&\text{ if and only if } %
		a\in \Max_{{\mathbf A}^{op}} L_{{\mathbf A}^{op}}\big(\{q(t)\mid q\in B%
		\text{ and }s\mathrel{R^{-1}}t\}\big)\\ 
			&\text{ if and only if } %
		a\in {}_{{\mathbf T}^{op}}G_{{\mathbf A}^{op}}(B)(s)\\
		&\text{ if and only if } %
		a\in \Min_{{\mathbf A}} L_{{\mathbf A}}\big(\{q(t)\mid q\in B%
		\text{ and }s\mathrel{R^{-1}}t\}\big)\\ 
		&\text{ if and only if } %
		a\in {}_{{\mathbf T}^{op}}F_{{\mathbf A}}(B)(s).\\
		\end{align*}
		It is enough to check that $(\mathcal P_+f)^T \circ {}_{\mathbf T}P_{\mathbf A}%
	={}_{\mathbf T}P_{\mathbf A'}\circ \mathcal P_+(f^T)$. The remaining cases dually follow.

Let $B\in\mathcal P_+(A^T)$, $a'\in A'$ and  $s\in T$. We put $a=f^{-1}(a')$ and $B'=(P_+(f^T))(B)$. We have that:
			\begin{align*}
	a'\in \big({}_{\mathbf T}P_{\mathbf A'}\circ \mathcal P_+(f^T)\big)(B)(s)&%
	\text{ if and only if } %
a'\in {}_{\mathbf T}P_{\mathbf A'}(B')(s)\\ 
			&\text{ if and only if } %
			a'\in \Min_{\mathbf A'} U_{\mathbf A'}\big(\{q'(t)\mid q'\in B'\text{ and }t\mathrel Rs\}\big)\\
			&\text{ if and only if } %
		a\in \Min_{\mathbf A} U_{\mathbf A}\big(\{q(t)\mid q\in B\text{ and }t\mathrel Rs\}\big)\\
			&\text{ if and only if } %
		a'\in \mathcal P_+f\big(\Min_{\mathbf A} U_{\mathbf A}\big(\{q(t)\mid q\in B\text{ and }t\mathrel Rs\}\big)\big)\\
		&\text{ if and only if } %
		a'\in \big(\big(\mathcal P_+f)^T \circ {}_{\mathbf T}P_{\mathbf A}\big)\big(B)(s).
		%
		\end{align*}
\end{proof}

\begin{remark}\label{duality}
Let us recall that if the MLUB-complete poset \(\mathbf{A}\) includes an antitone involution \({}'\), it is sufficient to define a pair of tense operators and obtain the second pair dually. Specifically, we obtain that \(P(B) = H(B')'\) and \(F(B) = G(B')'\) for each \(B \in \mathcal{P}_+(A^T)\). Here, \(X' = \{x' \mid x \in X\}\) for every \(X \in \mathcal{P}_+(A^T)\).

\end{remark}

Definition \ref{def:tnspr} presents a critical issue. In fact, in less general contexts, the tense operators are maps where the domain and codomain coincide. However, in this context, \(P(q)\) for \(q \in A^T\) need not be a single element of \(A^T\) (i.e., a singleton) but may be a non-empty subset of \(A^T\). Analogously for the remaining tense operators. This makes it impossible to directly compose two operators, or apply an operator twice. To overcome this problem, we will exploit the fact that \((\mathcal{P}_+(A))^T \subseteq \mathcal{P}_+(A^T)\), and we will define the composition of two tense operators through the following  ``transformation function'' \(\varphi\).

\begin{definition}\label{def:trsfm}
Let \(\mathbf{A} = (A, \leq)\) be an MLUB-complete poset, we define the \emph{transformation function}  $\varphi\colon(\mathcal P_+A)^T\rightarrow\mathcal P_+(A^T)$ as follows:
\[ \varphi(B):=\{q\in A^T\mid q(t)\in B(t)\text{ for all }t\in T\}=\bigtimes_{t\in T} B(t)=B\]
for all $B\in(\mathcal P_+A)^T$. 
\end{definition}

Lemma \ref{lemchl} illustrates some properties of the transformation function $\varphi$, which are useful for achieving our goals. Parts (i)-(iii) of Lemma \ref{lemchl} are well known, and the remaining ones follow immediately from Definition \ref{def:trsfm}.

\begin{lemma}\label{lemchl}{\rm(}cf.\ {\rm\cite{CL1})}
Let $\mathbf A=(A,\leq)$ be a poset, $T\neq\emptyset$, $p\in A^T$ and $B,C\in\mathcal (\mathcal P_+A)^T$. Let $\varphi$ as in Definition \ref{def:trsfm}. Then,
\begin{enumerate}[label={\rm{(\roman*)}}]
\item $\varphi$ is injective,
\item if $B(t)=\{p(t)\}$ for all $t\in T$ then $\varphi(B)=\{p\}$,
\item $B\leq C$ if and only if $\varphi(B)\leq\varphi(C)$, 
\item $B\leq_1 C$ if and only if $\varphi(B)\leq_1\varphi(C)$, 
\item $B\leq_2 C$ if and only if $\varphi(B)\leq_2\varphi(C)$, 
\item $B\sqsubseteq C$ if and only if $\varphi(B)\sqsubseteq\varphi(C)$, 
\item $B\subseteq C$ if and only if $\varphi(B)\subseteq\varphi(C)$. 
\end{enumerate}
\end{lemma}

Utilizing Definition \ref{def:trsfm} and the properties listed in Lemma \ref{lemchl}, we can define the composition between two tense operators.

\begin{definition}
	Let \(\mathbf{A} = (A, \leq)\) be an MLUB-complete poset. For any \(X, Y \in \{P, F, H, G\}\), where \(X, Y\) are maps \(X, Y \colon \mathcal{P}_+(A^T) \to (\mathcal{P}_+A)^T\), we define the composition of the two tense operators as follows:
	\[ X * Y := X \circ \varphi \circ Y, \]
	where \(\varphi\) is the transformation function as defined in Definition \ref{def:trsfm}.
\end{definition}

In Theorem \ref{prop3} and Theorem \ref{prop2}, we demonstrate how the tense operators behave when composed with the transformation function \(\varphi\).

\begin{theorem}\label{prop3}
	Let $(A,\leq)$ be an MLUB-complete  poset, $(T,R)$ a time frame with $R$ reflexive, and let $P$, $F$, $H$ and $G$ be operators induced by $(T,R)$. For any $q\in A^T$, we have the following
	\begin{align*}
	q & \leq(\varphi\circ P)(q), \\
	q & \leq(\varphi\circ F)(q), \\
	(\varphi\circ H)(q) & \leq q, \\
(\varphi\circ G)(q) & \leq q.
\end{align*}
\end{theorem}
\begin{proof}

	Let $s\in T$. We have
	\begin{align*}
			P(q)(s) & =\Min U(\{q(t)\mid t\mathrel Rs\}), \\
	(\varphi\circ P)(q) & =\{p\in A^T\mid p(t)\in\Min U(\{q(u)\mid u\mathrel Rt\})\text{ for all }t\in T\}, \\
q & \leq p\text{ for all }p\in(\varphi\circ P)(q), \\
q & \leq(\varphi\circ P)(q).
\end{align*}
The results for \(F\), \(G\), and \(H\) follow from the preceding discussion and Theorem \ref{propdual}, as we can express \(F\), \(G\), and \(H\) in terms of \(P\).

\end{proof}

\begin{theorem}\label{prop2}
Let \((A, \leq, {})\) be an MLUB-complete poset and \((T, R)\) a time frame. Furthermore, consider \(P\), \(F\), \(H\), and \(G\) operators induced by \((T, R)\), \(B \in (\mathcal{P}_+A)^T\), and \(s, t \in T\), and $\varphi$ the transformation function. Then the following hold:
	\begin{align*}
	P\big(\varphi(B)\big)(s) & =\Min U\big(\bigcup\{B(t)\mid t\mathrel Rs\}\big), \\
	F\big(\varphi(B)\big)(s) & =\Min U\big(\bigcup\{B(t)\mid s\mathrel Rt\}\big), \\
	H\big(\varphi(B)\big)(s) & =\Max L\big(\bigcup\{B(t)\mid t\mathrel Rs\}\big), \\
	G\big(\varphi(B)\big)(s) & =\Max L\big(\bigcup\{B(t)\mid s\mathrel Rt\}\big),
\end{align*}
\end{theorem}

\begin{proof}

	 Since Definition \ref{def:tnspr}  and Definition \ref{def:trsfm}, we have:
	
	\[ P\big(\varphi(B)\big)(s)=\Min U(\{p(t)\mid p\in\varphi(B)\text{ and }t\mathrel R s\})%
	=\Min U\big(\bigcup\{B(t)\mid t\mathrel R s\}\big).		\]
	
	The proof for $F$, $H$ and $G$ is analogous.
	
\end{proof}

Theorem \ref{thm:ordrs} shows several order relations among elements and elements to which one of the tense operators is applied, considering the different types of orderings presented in Definition \ref{dfn:rlzord}.

\begin{theorem}\label{thm:ordrs}
Let $(A, \leq, {})$ be an MLUB-complete poset and $(T, R)$ a time frame. Furthermore, consider $P$, $F$, $H$, and $G$ operators induced by $(T, R)$, let $C$ and $D$ be elements of $\mathcal P_+(A^T)$, and $\varphi$ the transformation function. Then the following hold:
\begin{enumerate}[label={\rm{(\roman*)}}]
\item If $C\subseteq D$ or $C\leq_1 D$ then $P(C)\leq_2 P(D)$ and $L(P(C))\subseteq L(P(D))$,  and $F(C)\leq_2F(D)$ and $L(F(C))\subseteq L(F(D))$, 
\item If $D\subseteq C$ or $C\leq_2 D$ then  $H(C)\leq_1H(D)$ and $U(H(D))\subseteq U(H(C))$, and $G(C)\leq_1G(D)$ and $U(G(D))\subseteq U(G(C))$,
\item $H(C)\leq P(C)$ and $G(C)\leq F(C)$,
\item $P(C)=P(LU(C))$, $F(C)=F(LU(C))$, $G(D)=G(UL(D))$, $H(D)=H(UL(D))$, 
\item if $R$ is reflexive then 
\begin{align*}
H(C)\leq C\leq P(C) \text{ and } G(C)\leq C\leq F(C).\tag{T3}
\end{align*}
\end{enumerate}
\end{theorem}

\begin{proof}
\begin{enumerate}[label={\rm{(\roman*)}}]
\item We have $C(t)\leq_1 D(t)$ for all $t\in T$. Hence also 
$\{q(t)\mid q\in C\text{ and }t\mathrel Rs\}\leq_1 \{p(t)\mid p\in D\text{ and }t\mathrel Rs\}$. From Lemma \ref{usefulprop}, (ii) we obtain that 
\begin{align*}
U(\{p(t)\mid p\in D\text{ and }t\mathrel Rs\})&\subseteq U(\{q(t)\mid q\in C\text{ and }t\mathrel Rs\})
\end{align*}
Hence, 
\begin{align*}
P(C)(s) & =\Min U(\{q(t)\mid q\in C\text{ and }t\mathrel Rs\})\\%
&\leq_2 \Min U(\{p(t)\mid p\in D\text{ and }t\mathrel Rs\})=	P(D)(s).
\end{align*}
Since \begin{align*}
L(P(D))(s)&=LU(\{p(t)\mid p\in D\text{ and }t\mathrel Rs\})&%
\supseteq LU(\{q(t)\mid q\in C\text{ and }t\mathrel Rs\})\\
&=L(P(C))(s)
\end{align*}
We also get that $L(P(C))\subseteq L(P(D))$.\\

Analogously, we obtain $F(C)\leq_2F(D)$ and $L(F(C))\subseteq L(F(D))$.

		\item It follows from duality by Theorem \ref{propdual}.
		
		\item Since $R$ is serial there exists some $u\in T$ with $u\mathrel Rs$. We have
				\begin{align*}
			H(C)(s) & =\Max L(\{q(t)\mid q\in C\text{ and }t\mathrel Rs\})\leq\{q(u)\mid q\in C\} \\
			& \leq\Min U(\{q(t)\mid q\in C\text{ and }t\mathrel Rs\})=P(C)(s).
		\end{align*}
		
		The second assertion follows analogously.
		
		\item We compute:
				\begin{align*}
			P(C)(s) & =\Min U(\{q(t)\mid q\in C\text{ and }t\mathrel Rs\})%
			=\Min ULU(\bigcup\{C(t)\mid t\mathrel Rs\})\\%
			& =\Min ULU(\bigcup\{LU(C(t))\mid t\mathrel Rs\})=%
			P(LU(C))(s).
		\end{align*}
		
		The remaining equalities dually follow from Theorem \ref{propdual}.
		
		\item If $R$ is reflexive we have that $s\mathrel Rs$. By the same procedure as in (iv) 
				we obtain that 
				\begin{align*}
			H(C)(s) & =\Max L(\{q(t)\mid q\in C\text{ and }t\mathrel Rs\})\leq\{q(s)\mid q\in C\}%
			=C(s) \\
			& \leq\Min U(\{q(t)\mid q\in C\text{ and }t\mathrel Rs\})=P(C)(s).
		\end{align*}
			\end{enumerate}
\end{proof}

\begin{theorem}\label{th2}
Let $(A,\leq)$ be an MLUB-complete  poset, $(T,R)$ a time frame with $R$ reflexive, and let $P$, $F$, $H$ and $G$ be operators {induced} by $(T,R)$. Then the following hold
\[\begin{array}{llll}
P\leq_2P*F, & F\leq_2F*P, & H\leq_1H*P, & G\leq_1G*P, \\
P*H\leq_2P, & F*H\leq_2F, & H\leq_1H*F, & G\leq_1G*F, \\
P*G\leq_2P, & F*G\leq_2F, & H*G\leq_1H, & G*H\leq_1G.
\end{array}	\]

\end{theorem}

\begin{proof}

According to Theorem~\ref{prop3}, for any $q\in A^T$, we have $q\leq(\varphi\circ F)(q)$. Hence, by (i) of Theorem ~\ref{prop2}, we obtain

\[	P(q)\leq_2P\big((\varphi\circ F)(q)\big)=(P\circ\varphi\circ F)(q)=(P*F)(q).	\]

This demonstrates $P\leq_2P*F$. The remaining inequalities can be proven similarly.

\end{proof}

\begin{theorem}\label{th1}

Let $(A,\leq,0,1)$ be an MLUB-complete poset, $(T,R)$ a time frame with reflexive $R$, and $P$, $F$, $H$, and $G$ operators induced by $(T,R)$. Let $X\colon\mathcal P_+(A^T)\to (\mathcal P_+A)^T$. Then
	\[	H*X,G*X\leq X\leq P*X,F*X, \]

	especially
	\[ 	P\leq P*P, F\leq F*F, H*H\leq H, G*G\leq G.\]

\end{theorem}

\begin{proof}

	If $q\in A^T$ and $s\in T$ then
	\[	(\varphi\circ X)(q)=%
	\{r\in A^T\mid r(u)\in X(q)(u)\text{ for all }u\in T\}=X(q) \]

	and hence
	\begin{align*}
				(P*X)(q)(s) & =\Min U(\bigcup\{X(q)(t)\mid t\mathrel Rs\})\geq X(q)(s), \\
(F*X)(q)(s) & =\Min U(\bigcup\{X(q)(t)\mid s\mathrel Rt\})\geq X(q)(s), \\
		(H*X)(q)(s) & =\Max L(\bigcup\{X(q)(t)\mid t\mathrel Rs\})\leq X(q)(s), \\
		(G*X)(q)(s) & =\Max L(\bigcup\{X(q)(t)\mid s\mathrel Rt\})\leq X(q)(s).
	\end{align*}
\end{proof}

That $P*P=P$ does not hold in general can be seen from the following example.

\begin{example}\label{ex1}

	Let  $A=\{0,a,b,c,d,e,f,g,1\}$ and $\mathbf A=(A,\leq,0,1)$ be the bounded poset (which is not a lattice) depicted in Figure~1:

	\vspace*{-2mm}

	\[	\setlength{\unitlength}{7mm}
\begin{picture}(6,9)
		\put(3,2){\circle*{.3}}

		\put(1,4){\circle*{.3}}

		\put(3,4){\circle*{.3}}

		\put(5,4){\circle*{.3}}

		\put(3,5){\circle*{.3}}

		\put(1,6){\circle*{.3}}

		\put(3,6){\circle*{.3}}

		\put(5,6){\circle*{.3}}

		\put(3,8){\circle*{.3}}

		\put(3,2){\line(-1,1)2}

		\put(3,2){\line(0,1)6}

		\put(3,2){\line(1,1)2}

		\put(1,6){\line(0,-1)2}

		\put(1,6){\line(1,-1)2}

		\put(1,6){\line(1,1)2}

		\put(5,6){\line(0,-1)2}

		\put(5,6){\line(-1,-1)2}

		\put(5,6){\line(-1,1)2}

		\put(3,6){\line(-1,-1)2}

		\put(3,6){\line(1,-1)2}

		\put(2.875,1.25){$0$}

		\put(.35,3.85){$a$}

		\put(3.4,3.85){$b$}

		\put(5.4,3.85){$c$}

		\put(3.4,4.85){$d$}

		\put(.35,5.85){$e$}

		\put(3.4,5.85){$f$}

		\put(5.4,5.85){$g$}

		\put(2.85,8.4){$1$}

		\put(2.3,0){{\rm Fig.~1}}
	\end{picture}	\]

	\bigskip

	\bigskip

Let $\mathbf T=(T,R)=(\{1,2,3\},\leq)$ be a time frame such that $1\leq 2\leq 3$. 

Let us consider  the propositions $p$ and $q$ defined by the following table 

\[ \begin{array}{r|l|l|l}

   t & 1  & 2  & 3 \\

\hline

p(t) & e & g & e \\

q(t) & f & f & g\\

r(t) & a & b & b

\end{array} \]

Then we have

\[ \begin{array}{r|l|l|l}

                                    t & 1         & 2         & 3 \\

\hline

                                 p(t) & e        & g        & e \\

                                 q(t) & f        & f        & g \\

                              G(p)(t) & b         & b         & e \\

                              G(q)(t) & \{b,c\}   & \{b,c\}   & g \\

                              H(p)(t) & e         & b   & b \\

                              H(q)(t) & f         & f         & \{b,c\}  \\

                              P(p)(t) & e         & 1         & 1 \\

                              P(q)(t) & f         & f         & 1 \\

                              P(r)(t) & a         & \{e,f\}   & \{e,f\} \\

                              (P*P)(r)(t) & a     & 1   & 1 \\

                               F(p)(t) & 1         & 1         & e \\

                              F(q)(t) & 1         & 1         & g \\

\end{array} \]

\end{example}

Another possible question regarding the tense operators is whether they form a dynamic pair. This concept was defined in \cite{CP15b}. For our purposes, it is necessary to adapt it to the context of posets. 

Definitions \ref{def:dymc} and \ref{def:dymc2} are expedient to introduce this concept.

\begin{definition}\label{def:dymc}
	Let $\mathbf{A}=(A,\leq,0,1)$ be a bounded poset, $T$ a set and a pair $(P,G)$ of mappings from $\mathcal{P}_+(A^T)$ to $(\mathcal{P}_+A)^T$. We call $(P,G)$ a \emph{dynamic pair} if for all $C,D\in \mathcal{P}_+(A^T)$, the following properties hold:
	\begin{enumerate}[label={\rm(T\arabic*)}, start=0]
		\item $G(1)=1$ and $P(0)=0$,
		\item 
		\begin{enumerate}[label={\rm(\roman*)}]
			\item if $C\leq_1 D$ then $P(C)\leq_2 P(D)$,
			\item if $C\leq_2 D$ then $G(C)\leq_1 G(D)$,
		\end{enumerate}
		\item $C\leq_1 (G*P)(C)$ and $(P*G)(C)\leq_2 C$.
	\end{enumerate}
\end{definition}

\begin{definition}\label{def:dymc2}
Let $\mathbf{A} = (A, \leq, 0, 1)$ be a bounded poset, $T$ a set, and $(P,F,H,G)$ a quadruple of mappings from $\mathcal{P}_+(A^T)$ to $(\mathcal{P}_+A)^T$. $\mathbf{A} = (A, \leq, 0, 1)$ is said to be a \emph{dynamic poset} if $(P,G)$ and $(F,H)$ form dynamic pairs.
\end{definition}
A fundamental concept is that of a Galois connection, which is essential in various areas of mathematics and computer science. It establishes a relationship between two partially ordered sets by linking their subsets systematically. In this case as well, we need a slight adjustment to adapt it to our context.
\begin{definition}\label{def:gls}
Let  $\mathbf A=(A,\leq,0,1)$ be a bounded poset and $T$ a set. We say that  mappings $P,G$   from $P_+(A^T)$ to $(\mathcal P_+A)^T$ form a {\em Galois connection} if for all $C,D\in P_+(A^T)$ the following condition holds:
\begin{align*}
	P(C)\leq_2 D\text{ if and only if } C\leq_1 G(D)\tag{Ga}.
\end{align*}	
\end{definition}

Theorem \ref{chaga} characterizes the concept of Galois connection (Definition \ref{def:gls}) in terms of a dynamic pair (Definition \ref{def:dymc}). In particular, it establishes an equivalence.

\begin{theorem}\label{chaga}

	Let  $\mathbf A=(A,\leq,0,1)$ be a bounded poset,  $T$ a set, and  $P,G$ mappings  from $P_+(A^T)$ to $(\mathcal P_+A)^T$. Then the following conditions are equivalent:

	\begin{enumerate}[label={{\arabic*.}}] 

	\item $P,G$ form a {Galois connection}.

	\item $P,G$ satisfy conditions (T1) and (T2).

\end{enumerate}

\end{theorem}

\begin{proof} Assume first that $P,G$ form a Galois connection. 
	
	Let $C,D\in P_+(A^T)$. Since $P(C)\leq_2 P(C)$ we have that 	
	\[C\leq_1 G(P(C))=(G*P)(C).\] 
	
	Similarly, $(P*G)(C)\leq_2 C$. Therefore, we conclude that (T2) holds. 

	Now, assume $C\leq_1 D$. From (T2) we obtain 
		\[D\leq_1 G(P(D))=(G*P)(D).\] 
		Since $\leq_1 $ is transitive we get $C\leq_1 G(P(D))$. 
		
		From (Ga) we conclude that $P(C)\leq_2 P(D)$. Similarly,  if $C\leq_2 D$ then $G(C)\leq_1 G(D)$.

	Suppose now that $P,G$ satisfy conditions (T1) and (T2). 
	
	Let  $C,D\in P_+(A^T)$ such that 	$P(C)\leq_2 D$. Then 
	\[C\leq_1(G*P)(C)\leq_1 P(D).\] 
	Since $\leq_1 $ is transitive we obtain that $C\leq_1 G(D)$. The remaining part follows from the same reasoning. 

\end{proof}

\begin{remark}\label{remzero} Recall that from {\rm (Ga)} we obtain $0\in P(0)$ and, similarly, $1\in G(1)$. Moreover, if $P(0)$ is an antichain  we conclude that $0= P(0)$. Similarly, if  $G(1)$ is an antichain, we have $1=G(1)$. 


\end{remark}

We conclude this section with Theorem \ref{tenseisdynamic}, which demonstrates that a bounded poset with tense operators as defined by (To) is indeed a dynamic one.

\begin{theorem}\label{tenseisdynamic}
Let $(A, \leq, 0, 1)$ be an MLUB-complete poset, $(T, R)$ a time frame with $R$ being a serial binary relation, and $P$, $F$, $H$, and $G$ operators induced by $(T, R)$. Then $\mathbf{A}$ and $(P, G)$, as well as $\mathbf{A}$ and $(F, H)$ form dynamic pairs.

\end{theorem}

\begin{proof} Let us notice it suffices to demonstrate that $\mathbf{A}$ and $(P, G)$ form a dynamic pair. The remaining part follows analogously.

Let $B,C\in \mathcal P_+(A^T)$, $p\in B$ and $s\in T$.

\begin{enumerate}[label={\rm{(T\arabic*)}}]



\item Suppose $B\leq_1 C$ then from Theorem \ref{prop2} we obtain $P(B)\leq_2 P(C)$.

 Similarly, if  $B\leq_2 C$ then $G(B)\leq_1 G(C)$.

\item Supoose $C\leq_2 D$, then we have
\begin{align*}
P(B)(s) & =\Min U(\{q(t)\mid q\in B \text{ and } t\mathrel Rs\}), \\
\varphi\big(P(B)\big) & =P(B)=\{r\in A^T\mid %
r(u)\in \Min U(\{q(t)\mid q\in B \text{ and }  t\mathrel Ru\})\\ &\phantom{=} \ \text{ for all }u\in T\}, \\
          (G*P)(B)(s) & =G\Big(\varphi\big(P(B)\big)\Big)(s)=\Max L(\{r(v)\mid r\in\varphi\big(P(B)\big)\text{ and }s\mathrel Rv\})= \\
  & =\Max L\big(\bigcup\{\Min U(\{B(t)\mid t\mathrel Rv\})\mid s\mathrel Rv\}\big), \\
   p(s) & \in L\big(\bigcup\{U(\{B(t)\mid t\mathrel Rv\})\mid s\mathrel Rv\}\big)\subseteq \\
  & \subseteq L\big(\bigcup\{\Min U(\{B(t)\mid t\mathrel Rv\})\mid s\mathrel Rv\}\big)
\end{align*}
and hence $p(s)\leq_1(G*P)(B)(s)$.

We conclude that  
\[p\leq_1(G*P)(B),\] 
i.e., $B\leq_1(G*P)(B)$. 
The remaining part of (T2) follows from Theorem \ref{propdual}.
Condition (T0) follows from Remark \ref{remzero} and Proposition \ref{chaga}.
\end{enumerate}
\end{proof}

Let us recall that if we assume in Theorem \ref{tenseisdynamic} that $R$ is reflexive, we obtain a dynamic poset with tense operators satisfying condition (T3) from Theorem \ref{prop2}.

\section{Tense operators in the Dedekind-MacNeille completion}\label{sec:tnsdmc}

It is well-known that every poset $(A,\leq)$ can be embedded into a complete lattice ${\mathbf L}$. This complete lattice captures all the information about joins (least upper bounds) and meets (greatest lower bounds) within the original poset.
We will frequently consider the Dedekind-MacNeille completion $\BDM(A,\leq)$ for this ${\mathbf L}$. 

Let ${\mathbf A} = (A,\leq)$ be a poset. Put $\DM({\mathbf A}):=\{B\subseteq A\mid LU(B)=B\}$. Then for $\DM({\mathbf A})=\{L(B)\mid B\subseteq A\}$, $\BDM({\mathbf A}):=(\DM({\mathbf A}),\subseteq)$ is 
a complete lattice and $x\mapsto L(x)$ is an embedding from $\mathbf A$ to $\BDM({\mathbf A})$ preserving all existing joins and meets,  and an order isomorphism between posets $\mathbf A$ and $(\{L(x)\mid x\in A\},\subseteq)$. We will usually identify $A$  with $\{ L(x) \mid x\in A\}$.

Consider  a complete lattice  $\mathbf L=(L;\leq,{}0, 1)$.   
Let  $(T,R)$ be a time frame. 
 Define the following mappings  $\widehat{G}, \widehat{P}, \widehat{H}$ 
 and $\widehat{F}$ on $\mathbf L^{T}$ as follows: 

\noindent{}for all $p\in L^{T}$, $\widehat{G}(p)$ is defined by   
        $$\mbox{$\widehat{G}(p)(s)=\bigwedge_{\mathbf L}\{p(t)\mid s R t\} $};\phantom{.}$$ 
\noindent{}for all $p\in L^{T}$, $\widehat{P}(p)$ is defined by 
$$\mbox{$ \widehat{P}(p)(s)=\bigvee_{\mathbf L}\{p(t)\mid t R  s\};\phantom{.}$}$$
\noindent{}for all $p\in L^{T}$, $\widehat{H}(p)$ is defined by
        $$\mbox{$\widehat{H}(p)(s)=\bigwedge_{\mathbf L}\{p(t)\mid t R s\}$};\phantom{.}$$ 
\noindent{}for all $p\in A$, $\widehat{F}(p)$ is defined by
$$\mbox{$ \widehat{F}(p)(s)=\bigvee_{\mathbf L}\{p(t)\mid s R  t\}$}.$$

The just constructed operator $\widehat{G}$ or $ \widehat{P}$ or $ \widehat{H}$ 
 or $\widehat{F}$ will be called an 
{\em operator on} $\mathbf L^{T}$ 
{\em constructed by means of the time frame} $(T,R)$.

We will now establish a useful connection between tense operators 
\(G\), \(P\), \(H\), and \(F\) from Section \ref{sectiontense} and 
operators $\widehat{G}, \widehat{P}, \widehat{H}$ 
 and $\widehat{F}$ on ${\DM({\mathbf A})}^{T}$.

 \begin{theorem}\label{thm:connect}
     Let $(A, \leq, {})$ be an MLUB-complete poset and $(T, R)$ a time frame. Let $\widehat{G}$, $ \widehat{P}$, $ \widehat{H}$ 
 and $\widehat{F}$ be {operators on} ${\DM({\mathbf A})}^{T}$ 
{constructed by means of the time frame} $(T,R)$. Furthermore, consider $G$, $P$, $H$, and $F$ operators induced by $(T, R)$, and let $C$ and $D$ be elements of $\mathcal P_+(A^T)$.  Then the following hold:
		
	\begin{enumerate}[label={\rm{(\roman*)}}]
		\item $\Max \widehat{G}(LU(C))=G(U(C))=G(\Min U(C))$;

  \item $\widehat{P}(L(D))=L(P(L(D))=L(P(\Max L(D))$;
		
		\item $\Max \widehat{H}(LU(C))=H(U(C))=H(\Min U(C))$;

  \item $\widehat{F}(L(D))=L(F(L(D))=L(F(\Max L(D))$.
  \end{enumerate}
 \end{theorem}
 \begin{proof} (i) It is enough to check that 
\[
\Max \widehat{G}(LU(C))(s)=G(U(C))(s)=G(\Min U(C))(s)
\]

for all $s\in T$. Let $s\in T$. Recall that 
\begin{align*}
 \widehat{G}(LU(C))(s)=&\bigcap\{LU(C)(t) \mid sRt\}\\
=&L(\bigcup\{U(C)(t) \mid sRt\})\\
=&L\big(\{q(t)\mid q\in U(C)\text{ and }s\mathrel Rt\}\big)
\end{align*} 

We conclude that 
$$\Max \widehat{G}(LU(C))=\Max L\big(\{q(t)\mid q\in U(C)\text{ and }s\mathrel Rt\}\big)=G(U(C)).$$

(ii) Similarly, it is enough to check that 
\[
\widehat{P}(L(D))(s)=L(P(L(D))(s)=L(P(\Max L(D))(s)
\]
 for all $s\in T$. Let $s\in T$. We compute: 
 \begin{align*}
 \widehat{P}(L(D))(s)&=\bigvee_{\mathbf {\DM({\mathbf A})}}\{L(D)(t)\mid t R  s\}=%
     LU\big(\{q(t)\mid q\in L(D)\text{ and }t\mathrel Rs\}\big)\\
     &=%
     L\big(\Min U\big(\{q(t)\mid q\in L(D)\text{ and }t\mathrel Rs\}\big)\big)%
     =L(P(L(D))(s).
 \end{align*}

Parts (iii) - (iv) can be proven similarly.

 \end{proof}

 \begin{corollary}\label{cor:connect}
     Let $(A, \leq, {})$ be an MLUB-complete poset and $(T, R)$ a time frame.  Let $\widehat{G}$, $ \widehat{P}$, $ \widehat{H}$ 
 and $\widehat{F}$ be {operators on} ${\DM({\mathbf A})}^{T}$ 
{constructed by means of the time frame} $(T,R)$. Furthermore, consider $G$, $P$, $H$, and $F$ operators induced by $(T, R)$, and let $C$ and $D$ be elements of $\mathcal P_+(A^T)$.  Then the following hold:

	\begin{enumerate}[label={\rm{(\roman*)}}]
		\item $\widehat{G}(L(C))=LU(G(UL(C)))=LU(G(C))$;

  \item $\widehat{P}(LU(D))=L(P(D))=L(P(LU(D)))$;
		
		\item $\widehat{H}(L(C))=LU(H(UL(C)))=LU(H(C))$;

  \item $\widehat{F}(LU(D))=L(F(D))=L(F(LU(D)))$.
  \end{enumerate}
 \end{corollary}
 \begin{proof} (i): From Theorem \ref{thm:ordrs}, (iv) we have
$G(UL(C))=G(C)$. Hence also 
$LU(G(C))=LU(G(UL(C)))$. Then by Theorem \ref{thm:connect}, (i) we obtain 
$LU(G(UL(C)))=\widehat{G}(L(C))$.

(ii): We compute,  using  Theorems \ref{thm:connect} and \ref{thm:ordrs}, (iv):
\begin{align}
\widehat{P}(LU(D))&= L(P(LU(D)))=L(P(D)). \label{supP}
\end{align}
 \end{proof}
 
\section{Constructions of a suitable preference relation}\label{sec:cnstrctns}


As shown in the Section \ref{sectiontense}, given a time frame \((T, R)\), one can define the tense operators \(P\), \(F\), \(H\), and \(G\) on the logic derived from a bounded poset, thereby obtaining a dynamic poset. However, the question arises whether, conversely, given a dynamic poset, one can find a time preference relation such that the induced tense operators coincide with the given ones. The following theorem demonstrates that this is indeed possible: from the given tense operators, we can construct a binary relation \(R^*\) on \(T\) such that the tense operators induced by the time frame \((T, R^*)\) are comparable to the given tense operators with respect to the quasi-order relations \(\leq_1\) and \(\leq_2\). Furthermore, if \(P\), \(F\), \(H\), and \(G\) are induced by some time frame, then the newly constructed tense operators are equivalent to the given ones via the equivalence relations \(\approx_1\) and \(\approx_2\) as described in Definition \ref{dfn:rlzord}.

\bigskip

Let us begin by providing a formal definition of a binary relation given a poset equipped with tense operators on a time set \(T\).

\begin{definition}
	Let $\mathbf{A} = (A, \leq, 0, 1)$ be a bounded poset, \(T\) a time set, and \(P\), \(F\), \(H\), and \(G\) tense operators on \(\mathbf{A}\). We define the relation \(R^*\) as the set of all pairs \((s, t) \in T^2\) satisfying the following conditions:
	\[
	H(B)(t) \leq B(s) \leq P(B)(t) \text{ and } G(B)(s) \leq B(t) \leq F(B)(s)
	\]
	for all \(B \in \mathcal{P}_+(A^T)\).
\end{definition}
We call \(R^*\) the \emph{relation induced by the tense operators \(P\), \(F\), \(H\), and \(G\)}.

To begin testing the suitability of this new relation \(R^*\), we will compare the original given tense operators with the poset and the tense operators derived from \(R^*\) according to Definition \ref{def:tnspr}.

\begin{theorem}\label{th4}
	Let $\mathbf{A}=(A,\leq,0,1)$ be an MLUB-complete bounded poset, $T$ a time set, and $P$, $F$, $H$, and $G$ be tense operators on $\mathbf{A}$. Furthermore, let $R^*$ denote the relation induced by $P$, $F$, $H$, and $G$, such that $R^*$ is serial. Let $P^*$, $F^*$, $H^*$, and $G^*$ be the tense operators on $\mathbf{A}$ induced by $(T,R^*)$. Then we have
	\[ 	P^*\leq_2 P, \quad F^*\leq_2 F, \quad H\leq_1 H^*, \quad \text{and} \quad G\leq_1 G^*. \]
\end{theorem}















\begin{proof}
	Let $B \in \mathcal{P}_+(A^T)$ and $s \in T$. Then we have
	\begin{align*}
		B(t) & \leq P(B)(s) \text{ for all } t \in T \text{ with } t \mathrel{R^*} s, \\
		B(t) & \leq F(B)(s) \text{ for all } t \in T \text{ with } s \mathrel{R^*} t, \\
		H(B)(s) & \leq B(t) \text{ for all } t \in T \text{ with } t \mathrel{R^*} s, \\
		G(B)(s) & \leq B(t) \text{ for all } t \in T \text{ with } s \mathrel{R^*} t.
	\end{align*}
	Thus, we have
	\begin{align*}
		P(B)(s) & \subseteq U\left(\bigcup\{B(t) \mid t \mathrel{R^*} s\}\right), \\
		F(B)(s) & \subseteq U\left(\bigcup\{B(t) \mid s \mathrel{R^*} t\}\right), \\
		H(B)(s) & \subseteq L\left(\bigcup\{B(t) \mid t \mathrel{R^*} s\}\right), \\
		G(B)(s) & \subseteq L\left(\bigcup\{q(t) \mid s \mathrel{R^*} t\}\right).
	\end{align*}
	Hence, we obtain
	\begin{align*}
		P^*(B)(s) & = \Min U\left(\bigcup\{B(t) \mid t \mathrel{R^*} s\}\right) \leq_2 P(B)(s), \\
		F^*(B)(s) & = \Min U\left(\bigcup\{B(t) \mid s \mathrel{R^*} t\}\right) \leq_2 F(B)(s), \\
		H(B)(s) & \leq_1 \Max L\left(\bigcup\{q(t) \mid t \mathrel{R^*} s\}\right) = H^*(B)(s), \\
		G(B)(s) & \leq_1 \Max L\left(\bigcup\{q(t) \mid s \mathrel{R^*} t\}\right) = G^*(B)(s).
	\end{align*}
\end{proof}

\begin{corollary}\label{corth4}
	Let $\mathbf A=(A,\leq,0,1)$ be an MLUB-complete bounded poset, $T$ a time set, and $P$, $F$, $H$, and $G$ tense operators on $\mathbf A$ satisfying condition (T3). Further, let $R^*$ denote the relation induced by $P$, $F$, $H$, and $G$, and $P^*$, $F^*$, $H^*$, and $G^*$ the tense operators on $\mathbf A$ induced by $(T,R^*)$. Then $R^*$ is reflexive and 
	\[ P^* \leq_2 P, \quad F^* \leq_2 F, \quad H \leq_1 H^*, \quad \text{and} \quad G \leq_1 G^*. \]
\end{corollary}

The rationale behind obtaining only inequalities between the operators $P$, $F$, $H$, $G$ and $P^*$, $F^*$, $H^*$, $G^*$ lies in the possibility that the given operators may be constructed in various ways, leading to a wide range of possibilities. Without specific information on their construction, they could exhibit diverse behaviors, possibly deviating from the standard expectations.

In the scenario where the given operators $P$, $F$, $H$, and $G$ are induced by a common time frame $(T,R)$ sharing the same time set $T$, we aim to investigate how these operators relate to each other. The following theorem demonstrates that under this circumstance, we indeed establish equivalences between them.

\begin{theorem}\label{th5}
Let $\mathbf A=(A,\leq,0,1)$ be an MLUB-complete bounded poset, $(T,R)$ a time frame with $R$ a serial binary relation, and $P$, $F$, $H$, and $G$ the tense operators on $\mathbf A$ induced by $(T,R)$. Further, let $R^*$ denote the relation induced by $P$, $F$, $H$, and $G$, and $P^*$, $F^*$, $H^*$, and $G^*$ the tense operators on $\mathbf A$ induced by $(T,R^*)$. Then $R\subseteq R^*$ and

\[	P^*= P, F^*= F, H^*=H\text{, and }G^*= G.	\]

\end{theorem}

\begin{proof} From Theorem \ref{tenseisdynamic} we know that $\mathbf A$ and $(P,G)$, and  $\mathbf A$ and $(F,H)$  form a dynamic pair, respectively.

Let $B\in  \mathcal P_+(A^T)$  and $s\in T$. If $t\in T$ and $s\mathrel Rt$ then
	\begin{align*}
	P(B)(t) & =\Min U(\bigcup\{B(u)\mid u\mathrel Rt\})\geq B(s), \\
F(B)(s) & =\Min U(\bigcup\{B(u)\mid s\mathrel Ru\})\geq B(t), \\
	H(B)(t) & =\Max L(\bigcup\{q(B)\mid u\mathrel Rt\})\leq B(s), \\
	G(B)(s) & =\Max L(\bigcup\{q(u)\mid s\mathrel Ru\})\leq B(t).
	\end{align*}

	Hence 
	\[ H(B)(t)\leq B(s)\leq P(B)(t)\text{ and }G(B)(s)\leq p(B)(t)\leq F(B)(s), 	\]

	i.e., $s{\mathrel R^*}t$. This shows $R\subseteq R^*$. Now
	\begin{align*}
P^*(B)(s) & =\Min U(\bigcup\{B(t)\mid t\mathrel{R^*}s\})\subseteq 
U(\{B(t)\mid t\mathrel{R^*}s\})\subseteq U(\{B(t)\mid t\mathrel Rs\}), \\
F^*(q)(s) & =\Min U(\bigcup\{B(t)\mid s\mathrel{R^*}t\})\subseteq U(\{B(t)\mid s\mathrel{R^*}t\})\subseteq U(\{B(t)\mid s\mathrel Rt\}), \\
H^*(q)(s) & =\Max L(\bigcup\{B(t)\mid t\mathrel{R^*}s\})\subseteq L(\{B(t)\mid t\mathrel{R^*}s\})\subseteq L(\{B(t)\mid t\mathrel Rs\}), \\
G^*(q)(s) & =\Max L(\bigcup\{B(t)\mid s\mathrel{R^*}t\})\subseteq L(\{B(t)\mid s\mathrel{R^*}t\})\subseteq L(\{B(t)\mid s\mathrel Rt\})
	\end{align*}

	and hence
	\begin{align*}
	P(B)(s) & =\Min U(\bigcup\{B(t)\mid t\mathrel Rs\})\leq_2P^*(B)(s), \\
	F(B)(s) & =\Min U(\bigcup\{B(t)\mid s\mathrel Rt\})\leq_2F^*(B)(s), \\
	H^*(B)(s) & \leq_1\Max L(\bigcup\{B(t)\mid t\mathrel Rs\})=H(B)(s), \\
	G^*(B)(s) & \leq_1\Max L(\bigcup\{B(t)\mid s\mathrel Rt\})=G(B)(s).
	\end{align*}

Therefore,
	\begin{align*}
	P  \leq_2 P^*, F  \leq_2 F^*, H^*  \leq_1 H \text{ and } G^*  \leq_1 G.
	\end{align*}
The rest follows from Theorem~\ref{th4} and Lemma \ref{usefulprop}, (iv).
\end{proof}

We pose the question of whether an alternative construction of the time preference relation \(R\) on the time set \(T\) exists, yielding equivalences between the given tense operators and those induced by \((T,R)\), akin to the results outlined in Theorem~\ref{th5}. We shall demonstrate the feasibility of such a construction, contingent upon the extension of the time set \(T\) in both past and future directions, accompanied by corresponding extensions of events. 

Before delving into the exposition of our construction, let us introduce some notation that will be useful for the subsequent discussion.

Let $\mathbf{A} = (A, \leq, 0, 1)$ be an MLUB-complete bounded poset, $T$ a time set, and $P$, $F$, $H$, and $G$ tense operators on $\mathbf{A}$. Define 

\[ \bar{T} := (T \times \{1\}) \cup T \cup (T \times \{2\}), \]

and let $R^*$ denote the relation induced by $P$, $F$, $H$, and $G$. Define 

\[ \bar{R} := \{((t,1),t) \mid t \in T\} \cup R^* \cup \{(t,(t,2)) \mid t \in T\}. \]

Furthermore, let $\bar{P}$, $\bar{F}$, $\bar{H}$, and $\bar{G}$ denote the tense operators induced by $(\bar{T}, \bar{R})$. For each $B \in (\mathcal{P}_+A)^T$, define $\bar{B}, \hat{B} \in (\mathcal{P}_+A)^{\bar{T}}$ such that

\begin{align*}
	\bar{B}((t,1)) & := \max L(P(B)(t)), \\
	\bar{B}(t) & := B(t), \\
	\bar{B}((t,2)) & := \max L(F(B)(t)),
\end{align*}
and
\begin{align*}
	\hat{B}((t,1)) & := \min U(H(B)(t)), \\
	\hat{B}(t) & := B(t), \\
	\hat{B}((t,2)) & := \min U(G(B)(t)),
\end{align*}

for all $t \in T$.

Let $B \in \mathcal{P}_+(A^T)$. Define $\bar{B} := \{\bar{p} \mid p \in B\}$ and $\hat{B} := \{\hat{p} \mid p \in B\}$.

\begin{theorem}\label{othercons}

Let $\mathbf{A}=(A,\leq,0,1)$ be an MLUB-complete bounded poset, $T$ a time set, and $P$, $F$, $H$, and $G$ tense operators. Define $\bar{T}$, $R^*$, $\bar{R}$, $\bar{P}$, $\bar{F}$, $\bar{H}$, $\bar{G}$, $\bar{B}$, and $\hat{B}$ as before. Then, for all $B \in \mathcal{P}_+(A^T)$,

\[	
\bar{P}(\bar{B})|T = P(B), \quad \bar{F}(\bar{B})|T = F(B), 
\bar{H}(\hat{B})|T = H(B), \text{ and } \bar{G}(\hat{B})|T = G(B).
\]

\end{theorem}

\begin{proof}
Let $B\in  \mathcal P_+(A^T)$ and $s\in T$. Then $H(B)(s)\leq B(t)\leq P(B)(s)$ and hence
 \[\hat B\big((s,1)\big)\leq B(t)\leq\bar B\big((s,1)\big),\]
for all $t\in T$ with $t\mathrel{R^*}s$. Similarly $G(B)(s)\leq B(t)\leq F(B)(s)$ implying 
\[\hat B\big((s,2)\big)\leq B(t)\leq\bar B\big((s,2)\big)\]
for all $t\in T$ with $s\mathrel{R^*}t$. Therefore
	\begin{align*}
		\bar P(\bar B)(s) & =\Min U\big(\bigcup\{\bar B(\bar t)\mid\bar t\mathrel{\bar R}s\}\big)=%
		\Min U\Big(\bar B\big((s,1)\big)\cup\bigcup\{\bar B(t)\mid t\mathrel{R^*}s\}\Big)\\
		& =\Min U\Big(\Max L(P(B)(t))\Big)\cap U\Big(\bigcup \{B(t)\mid t\mathrel{R^*}s\}\Big)\\ 
		&=\Min U\Big(\bigcup \{B(t)\mid t\mathrel{R^*}s\}\Big)= P(B)(s).
		\end{align*}
\begin{align*}
		\bar F(\bar B)(s) & =\Min U\big(\bigcup\{\bar B(\bar t)\mid s\mathrel{\bar R}\bar t\}\big)=%
		\Min U\Big(\bigcup\{\bar B(t)\mid s\mathrel{R^*}t\}\cup\bar B\big((s,2)\big)\Big)\\
		& =\Min U\Big(\bigcup\{B(t)\mid s\mathrel{R^*}t\}\Big)\cap%
		U\Big(\Max L(F(B)(t))\big)\Big)\\
		&=\Min U\Big(\bigcup\{B(t)\mid s\mathrel{R^*}t\}\Big)= F(B)(s).
\end{align*}
\begin{align*}
\bar H(\hat B)(s) & =\Max L\big(\bigcup\{\hat B(\bar t)\mid\bar t\mathrel{\bar R}s\}\big)=%
		\Max L\Big(\hat B\big((s,1)\big)\cup\{\hat B(t)\mid t\mathrel{R^*}s\}\Big)\\
		& =\Max L\Big(\Min U(H(B)(s))\Big)\cap L\Big(\bigcup\{B(t)\mid t\mathrel{R^*}s\}\Big)\\%
		&=\Max L\Big(\bigcup\{B(t)\mid t\mathrel{R^*}s\}\Big)= H(B)(s).
 \end{align*}

\begin{align*}		
\bar G(\hat B)(s) & =\Max L\big(\bigcup\{\hat B(\bar t)\mid s\mathrel{\bar R}\bar t\}\big)=%
	\Max L\Big(\bigcup\{\hat B(t)\mid s\mathrel{R^*}t\}\cup\hat B\big((s,2)\big)\Big)\\
	& =\Max L\Big(\bigcup\{B(t)\mid s\mathrel{R^*}t\}\Big)\cap %
	L\Big(\Min U(G(B)(s))\Big)\\
		&= \Max L\Big(\bigcup\{B(t)\mid s\mathrel{R^*}t\}\Big)= G(B)(s).
	\end{align*}
\end{proof}

\section{Logical connectives derived by posets}\label{sec:mixed}

In this section, we explore the introduction of the ``unsharp'' or ``inexact'' logical connectives, namely conjunction $\odot$ and implication $\to$, within the context of a poset $(A, \leq)$ that satisfies both the Antichain Condition (ACC) and the Descending Chain Condition (DCC). Our aim is to demonstrate how these connectives form a sort of adjoint pair. Furthermore, we extend the definition of these connectives to subsets of a poset, with the objective of establishing a relationship with tense operators. We conclude the section by presenting several results regarding this relationship.

\begin{definition}\label{def:lc}
Let $(A, \leq)$ be a poset satisfying both the Antichain Condition (ACC) and the Descending Chain Condition (DCC). We define the following binary operations:
\begin{align*}
	\begin{array}{r@{\,\,}l}
x\odot y&:=\Max	L(x,y),\\
x\to y&=:\Max \{z\in A\mid x\odot z\leq y\}. 
	\end{array}\tag{*}
\end{align*}
for every $x, y \in A$.
\end{definition}

Let us outline some important observations about Definition \ref{def:lc}. The results of $x\odot y$ or $x\to y$ may not be elements of $A$ but rather subsets of $A$ that are antichains. Additionally, $x\odot y=y\odot x$, which means that the operation $\odot$ is commutative. Finally, if the poset $(A,\leq,0)$ includes a least element $0$, we can define a negation as:

\[
\neg x= x\to 0. 
\]

More precisely, these logical connectives $\odot$ and $\rightarrow$ are binary operators on $A$, as they are mappings from $A^2$ to $\mathcal{P}_+A$. We extend them to $(\mathcal{P}_+A)^2$ by defining them as follows:
\begin{align*}
\begin{array}{r@{\,\,}l}
	B\odot C & :=\Max\{a\in A\mid a\leq U\big(\bigcup \{b\odot c \mid b\in B,c\in C\}\big)\}, \\
	B\rightarrow C & :=\Max\{a\in A\mid a\odot B \leq U(C)\}
	\end{array}\tag{**}
\end{align*}

for all $B,C\in \mathcal{P}_+A$. 

These extended operators $\odot$ and $\rightarrow$ are mappings from $(\mathcal{P}_+A)^2$ to $\mathcal{P}_+A$, and even in this case, they exhibit commutativity: $B\odot C=C\odot B$.

Finally, we can adapt the extended operators $\odot$ and $\rightarrow$ to the case where $B,C\in(\mathcal{P}_+A)^T$, with $(A, \leq)$ a poset and $T$ a time set. For all $B,C\in(\mathcal P_+A)^T$ and all $t\in T$, we have
\begin{equation*}\begin{split}\label{***}
	(B\odot C)(t) & :=B(t)\odot C(t), \\
	(B\rightarrow C)(t) & :=B(t)\rightarrow C(t)
	\end{split}\tag{***}
\end{equation*}

Lemma \ref{propim} and Lemma \ref{propim2} illustrate some properties of the pair of operators $\odot$ and $\rightarrow$ just defined. In the first case, we consider their operation on elements of a poset, while in the second case, on subsets of the same poset.

\begin{lemma}\label{propim}
	Let $(A,\leq,0,1)$ be a bounded poset satisfying both the {\rm ACC} and the {\rm DCC}, and let $x,y,z, w\in A$. Then, the following properties hold:

	\begin{enumerate}[label={\rm{(\arabic*)}}]

		\item $x\odot y=y\odot x$ and $1\odot x=x=x\odot 1$, 

		\item $x\odot (y\odot z)=(x\odot y)\odot z$, 

		\item if $x\leq z$ and $y\leq w$ then $x\odot y \leq_1 z\odot w$, 

			\item $x\odot y \leq_1 z$ if and only if $x\leq_1 y\to z$, 	\quad\quad\quad\quad\quad (adjointness)

		\item $(x\to y)\odot x \leq y$,

	\end{enumerate}

\end{lemma}

\begin{proof}



	\begin{enumerate}[label={\rm{(\arabic*)}}]

		\item The commutativity $x\odot y=y\odot x$ follows directly from the definition of $\odot$. For the second part, we apply the definition of $\odot$ and calculate:
			\begin{align*}
	1\odot x=\Max	L(1,x)= \Max	L(x)=x. 
		\end{align*}

	\item To show the associativity of $\odot$, it is sufficient to apply the definition and calculate directly:
	\begin{align*}
		x\odot (y\odot z)&=x\odot \Max	L(y,z)= \Max	L(x,y,z), \\
		(x\odot y)\odot z&=\Max	L(x,y)\odot z= \Max	L(x,y,z). 
	\end{align*}
Thus, $x \odot (y \odot z) = (x \odot y) \odot z$, demonstrating the associativity of $\odot$.
		
		\item 	Suppose \(x \leq z\) and \(y \leq w\). Then, we have:
		
		\[ L(x, y) \subseteq L(z, w). \]
	
	This is because any lower bound of \(x\) and \(y\) is also a lower bound of \(z\) and \(w\), given the assumptions \(x \leq z\) and \(y \leq w\). Therefore, we have:
		
		\[ \Max L(x, y) \leq_1 \Max L(z, w). \]
		
		By the definition of \(\odot\), it follows that:
		
		\[ x \odot y = \Max L(x, y) \leq_1 \Max L(z, w) = z \odot w. \]

		\item Suppose \( x \leq_1 y \to z \). By definitions of \(\to\) and $\leq_1$, there exists \( u \in \Max \{ w \in A \mid y \odot w \leq z \} \) such that \( x \leq u \).
		
		Since \( x \leq u \), we have \( L(x, y) \subseteq L(u, y) \). Given that \( u \in \Max \{ w \in A \mid y \odot w \leq z \} \), it follows that:
				\[
		L(u, y) \leq_1 z.
		\]
		
		Hence:
				\[
		L(x, y) \subseteq L(u, y) \leq_1 z.
		\]
		
		Therefore:
				\[
		x \odot y = \Max L(x, y) \leq_1 z.
		\]
		
		2. Conversely, suppose \( x \odot y \leq_1 z \). By definitions of \(\odot\) and $\leq_1$, this means:
				\[
		\Max L(x, y) \leq_1 z
		\]
		
		Then there exists \( u \in \Max \{ w \in A \mid y \odot w \leq z \} \) such that \( x \leq u \). Since \( u \) is the maximal element for which \( y \odot w \leq z \), we have:
				\[
		x \leq u.
		\]
		
		Therefore:
				\[
		x \leq_1 y \to z
		\]

		\item Suppose \( z \in (x \to y) \odot x \). By the definition of \(\odot\), there exists \( w \in A \) such that \( w \in x \to y \) and \( z \in \Max L(w, x) \).
		 
		 By the definition of \( x \to y \), \( w \in x \to y \) implies \( x \odot w \leq y \). 
		 
		 Since \( z \in \Max L(w, x) \), we have:
		
		\[
		z \leq \Max L(w, x) = x \odot w \leq y
		\]
		
		Therefore, \( z \leq y \).

	\end{enumerate}	

\end{proof}

\begin{lemma}\label{propim2}

	Let $(A,\leq,0,1)$ be a bounded poset satisfying both the {\rm ACC} and the {\rm DCC}, and let $B,C,D, E\subseteq A$. Then, the following properties hold:

	\begin{enumerate}[label={\rm{(\arabic*)}}]

		\item $B\odot C=C\odot B$ and $1\odot B=B\odot 1=\Max  LU(B)$, 

		\item if $D\leq_1 B$ and $D\leq_1 C$ then $D \leq_1 B\odot C$, 

		\item $B\odot C \leq \Min U(B)$ and $B\odot C \leq_1 \Max LU(B)$,

		\item $B\odot B =\Max  LU(B)$,

		\item if $B\leq_1 D$ and $C\leq_1 E$ then $B\odot C \leq_1 D\odot E$,

	\end{enumerate}

\end{lemma}

\begin{proof} 
\begin{enumerate}[label={\rm{(\arabic*)}}]
	\item The fact that $B\odot C=C\odot B$ follows from the definition of $\odot$ and Lemma 
\ref{propim}, (i). Now we compute:
\begin{align*}
	1\odot B=\Max  LU\big(\bigcup \{1\odot b \mid b\in B\}\big)%
	=\Max  LU\big(\bigcup \{\{b\} \mid b\in B\}\big)=\Max  LU(B). 
	\end{align*}

	\item Suppose  that $D\leq_1 B$ and $D\leq_1 C$ and let $a\in D$. Then there exist $\overline{b}\in B$ and $\overline{c}\in C$ such that $a\leq \overline{b}$ and $a\leq \overline{c}$. Hence $a\leq u$ for some $u\in \overline{b}\odot \overline{c}$. Consequently, we have:
			\[u\in \overline{b}\odot \overline{c}\subseteq %
			 LU\big(\bigcup \{b\odot c \mid b\in B,c\in C\}\big).\] 

			 Therefore we can find 
			 \[d\in \Max LU\big(\bigcup \{b\odot c \mid b\in B,c\in C\}\big)=B\odot C\] 
			 such that $a\leq u\leq d$. 

		\item Let $a\in B\odot C$. We have 
				\[a\in B\odot C \subseteq LU\big(\bigcup \{b\odot c \mid b\in B,c\in C\}\big)\subseteq %
	LU(B)\] 
	
	and $d\in \Min U(B)=ULU(B)$. Then $a\leq d$. Since $B\odot C\subseteq LU(B)$  we obtain 
		\[B\odot C \leq_1 \Max LU(B).\] 

	\item By definition of $\odot$, we have 
		\[B \odot B = \Max LU\left(\bigcup {b \odot c \mid b, c \in B}\right) = \Max LU(B).\]

		\item Suppose that $B\leq_1 D$ and $C\leq_1 E$, and  let $b\in B,c\in C$ and $a\in b\odot c$. Then there exist $d\in D$ and $e\in E$ such that $b\leq d$ and $c\leq e$. Hence, there exists $u\in d\odot e$ such that $a\leq u$. This implies
		\begin{align*}U\big(\bigcup \{d\odot e \mid b\in D,c\in E\}\big)&\subseteq%
	U\big(\bigcup \{b\odot c \mid b\in B,c\in C\}\big), \text{ i.e.}, \\
	\Max LU\big(\bigcup \{b\odot c \mid b\in B, c\in C\}\big)&\subseteq%
	LU\big(\bigcup \{d\odot e \mid b, d\in D, e\in E\}\big). 
	\end{align*}

	We conclude $B\odot C \leq_1 D\odot E$. 

		\end{enumerate}	

\end{proof}		

To investigate the relationship between the tense operators defined in Section \ref{sectiontense} and the logical connectives $\odot$ and $\rightarrow$, we can analyze Theorem \ref{PGFHpropim}, which illustrates some of their relationships.

\begin{theorem}\label{PGFHpropim}

Let $(A,\leq,0,1)$ be a bounded poset satisfying both the Antichain Condition (ACC) and the Descending Chain Condition (DCC), and let $(T,R)$ be a serial time frame. $P$, $F$, $H$, and $G$ denote the tense operators on $\mathbf{A}$ induced by $(T,R)$.

For $B, C\in \mathcal{P}_+(A^T)$, the following relationships hold:
	\begin{enumerate}[label={\rm{(\arabic*)}}]

		\item $P(B\odot B )=P(B)$ and $F(B\odot B )=F(B)$,

		\item $\Max L\big(P(B\odot C)\big)\leq_1 \Max L\big(P(B)\big)\odot %
\Max L\big(P(C)\big)$,

		\item $\Max L\big(F(B\odot C)\big)\leq_1 \Max L\big(F(B)\big)\odot %
		 \Max L\big(F(C)\big)$,

	\end{enumerate}

\end{theorem}

\begin{proof}

	\begin{enumerate}[label={\rm{(\arabic*)}}] 

		\item  By Theorem \ref{prop2}, (iv), we know that 

		$P(LU(B))=P(B)$. Since $B\odot B=\Max  LU(B)$ by Theorem \ref{propim2}, (iii), 

		we conclude $$P(B\odot B)=P(\Max  LU(B))= P(LU(\Max  LU(B)))=P(LU(B))=P(B).$$

		By the same procedure, we show that $F(B\odot B )=F(B)$.

		\item From Lemma \ref{propim2}, (2) and (3) we know that 

		$B\odot C \leq_1 \Max LU(B)=B\odot B$. Hence 

		$P(B\odot C) \leq_2 P(B\odot B)=P(B)$. Similarly,  $P(B\odot C) \leq_2 P(C)$.

		From Lemma \ref{usefulprop}, (iii) we obtain that 

		$L\big(P(B\odot C)\big) \subseteq L\big(P(B)\big)$, i.e., 

		$\Max L\big(P(B\odot C)\big)\leq_1 \Max L\big(P(B)\big)$. Therefore 

		$\Max L\big(P(B\odot C)\big)\leq_1 \Max L\big(P(C)\big)$. 

		Now by applying Lemma \ref{propim2}, (1), we conclude that 

		$\Max L\big(P(B\odot C)\big)\leq_1 \Max L\big(P(B)\big)\odot %
\Max L\big(P(C)\big)$.

		\item It follows immediately from (ii) applied to $R^{-1}$.

\end{enumerate}	

\end{proof}		

\section{A poset with given logical connectives}\label{sec:cnnctvs}

In Section \ref{sec:mixed}, we introduced certain ``mixed'' logical connectives and investigated how they relate to the particular tense operators of Definition \ref{def:tnspr}. This construction is possible in a wide range of poset classes, but the question remains: how do they relate to the tense operators \(P, F, H\), and \(G\)? An answer to this issue in our context is provided in Theorem \ref{PGFHpropim}. However, this relationship appears to be somewhat complex. The reason is that these logical connectives allow too much freedom, meaning they do not determine their values with exact precision. Therefore, we will now shift our focus to the case where these connectives are already defined in such a way that our structure forms a residuated poset, and then investigate their relationship with the tense operators.

We will begin by providing the formal notion of a residuated poset and then demonstrate some of its basic properties. Subsequently, we will present a method to combine the application of logical connectives with tense operators. Indeed, in this case, we encounter the problem of needing to generalize operations between elements to operations between subsets.

\begin{definition}
	A poset $(A, \leq, \odot, \rightarrow, 0, 1)$ is a \emph{residuated poset} if the following  conditions are satisfied:
	\begin{enumerate}
		\item $x \odot 1 = x$ for all $x \in A$;
		\item $x \odot y = y \odot x$ for all $x, y \in A$;
		\item $x \odot (y \odot z) = (x \odot y) \odot z$ for all $x, y, z \in A$;
		\item If $x \leq v$ and $y \leq z$, then $x \odot y \leq v \odot z$ for all $x, y, v, z \in A$;
		\item $x \odot y \leq z$ if and only if $x \leq y \rightarrow z$ for all $x, y, z \in A$.
	\end{enumerate}
\end{definition}

Let us state without proof a simple but useful lemma.

\begin{lemma}\label{lmm:bs}
	Let $(A, \leq, \odot, \rightarrow, 0, 1)$ be a residuated poset. Then the following inequalities hold:
	\begin{enumerate}
		\item $x \le y \rightarrow (x \odot y)$;
		\item $(x \rightarrow y) \odot x \le y$.
	\end{enumerate}
\end{lemma}

Moreover, if $(A, \leq, \odot, \rightarrow, 0, 1)$ is  a residuated poset 
then $(\DM({\mathbf A}),\subseteq, {\bm{\odot}}, {\bm{\rightarrow}}, %
{\bm{0}}, {\bm{1}})$ inherits this structure and becomes a complete residuated lattice  (see e.g. \cite{ZP} or \cite{GJKO}) such that, for all $X,Y\in \DM({\mathbf A})$,
\begin{align*}
    X{\bm{\odot}} Y&= LU(\{x\odot y\mid x\in X, y\in Y\},&{\bm{0}}=LU(\emptyset),\\
   X {\bm{\rightarrow}} Y&%
   =\bigcap\{L(x\rightarrow z)\mid x\in X, z\in U(Y)\},&
   {\bm{1}}=LU(A).
\end{align*}
Moreover, for any two subsets $X,Y\subseteq A$ we have that 
\begin{align}
    LU(X){\bm{\odot}} LU(Y)&= LU(\{x\odot y\mid x\in X, y\in Y\},&\\
   LU(X) {\bm{\rightarrow}} LU(Y)&%
   =\bigcap\{L(x\rightarrow z)\mid x\in X, z\in U(Y)\}.&
\end{align}
To explore the relationship between these logical connectives and tense operators, we need to define how they can be combined. It is important to note that the tense operators \(P\), \(F\), \(H\), and \(G\) map from \(\mathcal{P}_+(A^T)\) to \((\mathcal{P}_+A)^T\), while the logical connectives are mappings from \(A^2\) to \(A\). Therefore, we cannot directly apply the logical connectives to the images of \(P\), \(F\), \(H\), and \(G\).


\begin{definition}\label{def:lgets}
	Let \(\mathbf{A} = (A, \leq, \odot, \rightarrow, 0, 1)\) be an MLUB-complete bounded residuated poset. We extend the logical connectives for \(B, C \in \mathcal{P}_+(A)\) as follows:
		\begin{align*}
		B \boxdot C &=  
  \Max LU(\lbrace b \odot c \mid b \in B \, \text{ and } c \in C \rbrace)\\[0.2cm]
		B \Rightarrow C&= 
  \Min U(\bigcap\{L(b\rightarrow d)\mid b\in B, d\in U(C)\})%
  = \Min U(LU(B) {\bm{\rightarrow}} LU(C)).
	\end{align*}
\end{definition}

We observe that if $A$ is a residuated poset, then $A^T$ is also a residuated poset, with the operations computed component-wise. Moreover, the previous definitions can be utilized even in the case of $B, C \in \mathcal{P}_+(A^T)$, as presented in Section \ref{sec:mixed}.

Theorem \ref{thm:adj} demonstrates that Definition \ref{def:lgets} is fully functional for our purposes, as it continues to preserve a certain form of residuation.
\begin{theorem}\label{thm:adj}
	Let $(A, \leq, \odot, \rightarrow, 0, 1)$ be 
 an MLUB-complete bounded residuated poset, and let $B, C, D \in \mathcal{P}_+(A)$. Consider the extended logical connectives \(\boxdot\) and \(\Rightarrow\). Then the following properties hold:
	\begin{enumerate}
		\item $B \boxdot \{1\} = \Max  LU (B)$ and  
  $ \{1\} \Rightarrow B= \Min  U (B)$;
		\item $B \boxdot C = C \boxdot B$;
		\item if $B \le U(C)$ and $ D \le U(E)$ then 
  $B \boxdot D \le U(C \boxdot E)$;
		\item $B \boxdot C \le U(D)$ if and only if 
  $B \le C \Rightarrow D$;
  \item $L(C \Rightarrow D)\boxdot C\leq U(D)$ and 
  $C\leq L(C \Rightarrow D)\Rightarrow D$.
	\end{enumerate}
\end{theorem}

\begin{proof}
	(1)	$B \boxdot 1 = \Max LU(\lbrace b \odot 1 \mid b \in B  \rbrace)= \Max \ LU (B)$.

 Similarly, $ \{1\} \Rightarrow B = \Min U(LU(\{1\}) {\bm{\rightarrow}} LU(B))= \Min U(A {\bm{\rightarrow}} LU(B))= \Min  U (B)$.
	
	(2) This is straightforward.
	
	(3) Suppose $B \le U(C)$ and $ D \le U(E)$. Then  
 $B \subseteq LU(C)$ and $ D \subseteq LU(E)$.
Hence  
		\[
	\left\lbrace b \odot d | b \in B, d \in D \right\rbrace \subseteq 
 LU(\left\lbrace c \odot e | c \in C, e \in E \right\rbrace). 
	\]
 We conclude 
 \[
	LU(\left\lbrace b \odot d | b \in B, d \in D \right\rbrace) \subseteq 
 LU(\left\lbrace c \odot e | c \in C, e \in E \right\rbrace). 
	\]
	Therefore,
	\begin{equation*}\begin{split}
				B \boxdot D &%
    =\Max LU(\lbrace b \odot d \mid b \in B \ \text{ and } d \in D \rbrace)
    \subseteq 
 LU(\left\lbrace c \odot e | c \in C, e \in E \right\rbrace)\\ 
 &=LULU(\left\lbrace c \odot e | c \in C, e \in E \right\rbrace)=%
 LU(C \boxdot E).
				\end{split}
	\end{equation*}
We obtain that $B \boxdot D \le U(C \boxdot E)$.
	
	(4) We compute: 
 \begin{align*}
     B \boxdot C \le U(D)&\text{ iff } %
     LU(B){\bm{\odot}} LU(C)\subseteq LU(D) \text{ iff } %
     LU(B)\subseteq LU(C) {\bm{\rightarrow}} LU(D)\\
     &\text{ iff } B\subseteq LU(C) {\bm{\rightarrow}} LU(D)%
     = L(\Min U(LU(C) {\bm{\rightarrow}} LU(D)))\\
     &\text{ iff }  B\leq \Min U(LU(C) {\bm{\rightarrow}} LU(D)))=%
     C \Rightarrow D.
 \end{align*}

 (5) Clearly, $L(C \Rightarrow D)\leq C \Rightarrow D$. Hence from (4) we conclude that $L(C \Rightarrow D)\boxdot C\leq U(D)$. Again from (4) and (2) we 
 obtain that $C\leq L(C \Rightarrow D)\Rightarrow D$.
\end{proof}

To leverage Theorem \ref{thm:adj} and the properties of tense operators, we can establish the following significant relationship in Theorem \ref{proptensers}. 

To do this, we will need the following 
\cite[Theorem 10.3, Lemmas 10.1 and 10.2]{CP15a}.

\begin{theorem}\label{fuzzthupcomplate} 
Let \/ $\mathbf{L}=(L, \leq, \odot, \rightarrow, 0, 1)$ be a 
complete residuated lattice  and let $(T, R)$ be a  time frame. Let 
$\widehat{G}, \widehat{F}$ (or $\widehat{H}, \widehat{P}$) 
be operators on ${\mathbf L}^T$
constructed by means of the time frame $(T,R)$.
  Then $(\widehat{P},\widehat{G})$ and $(\widehat{F}, \widehat{H})$ are Galois connections
such that, 
 for all $p_1, p_2\in L^T$ and all $q_1, q_2\in  L^T$, 

\begin{enumerate}[leftmargin=1.52cm]
 \item[{\rm(DR1)}] $\widehat{G}(p_1)\odot \widehat{G}(p_2) \leq \widehat{G}(p_1\odot p_2)$ and\ \ 
      $\widehat{G}(p_1\rightarrow p_2)\leq %
      \widehat{G}(p_1)\rightarrow  \widehat{G}(p_2)$,
    
  \item[{\rm(DR2)}] $\widehat{P}(q_1\odot q_2) \leq \widehat{P}(q_1)\odot \widehat{P}(q_2)$,
      \item[{\rm(DR3)}] $\widehat{H}(q_1)\odot \widehat{H}(q_2) \leq \widehat{H}(q_1\odot q_2)$  and\ \ 
      $\widehat{H}(p_1\rightarrow p_2)\leq %
      \widehat{H}(p_1)\rightarrow  \widehat{H}(p_2)$,
  \item[{\rm(DR4)}] $\widehat{F}(p_1\odot p_2) \leq \widehat{F}(p_1)\odot \widehat{F}(p_2) $. 
  \end{enumerate}
\end{theorem}

How the operators $G$, $P$, $H$ and $F$ are related with our operations $\boxdot$ and $\Rightarrow$ 
defined for subsets is described in the following result.

\begin{theorem}\label{proptensers}
Let $(A, \leq, \odot, \rightarrow, 0, 1)$ be 
 an MLUB-complete bounded residuated poset, $\mathbf T=(T,R)$ a time frame. Then, 
 for all $C_1, C_2\in \mathcal{P}_+(A^T)$ and all 
$D_1, D_2\in  \mathcal{P}_+(A^T)$, 

\begin{enumerate}[leftmargin=1.52cm]
\item[{\rm(DT1)}] ${G}(U(C_1))\boxdot {G}(U(C_2))\leq \Min U(G(U(C_1\boxdot C_2))$ and 
\item[\phantom{{\rm(DT1)}}] ${G}(B_1\Rightarrow B_2)
         \leq %
         {G}(U(B_1)) \Rightarrow {G}(U(B_2));$
\item[{\rm(DT2)}]  $L({P}(D_1 \boxdot D_2)) \leq U(L({P}(D_1)){\boxdot} L({P}(D_2)))$;
\item[{\rm(DT3)}] ${H}(U(C_1))\boxdot {H}(U(C_2))\leq \Min U(H(U(C_1\boxdot C_2))$ and 
\item[\phantom{{\rm(DT3)}}] ${H}(B_1\Rightarrow B_2)
         \leq %
         {H}(U(B_1)) \Rightarrow {H}(U(B_2));$
\item[{\rm(DT4)}]  $L({F}(D_1 \boxdot D_2)) \leq U(L({F}(D_1)){\boxdot} L({F}(D_2)))$;
  \end{enumerate} 
\end{theorem}
\begin{proof} (DT1): From Theorems \ref{thm:connect} and \ref{fuzzthupcomplate} we obtain by (DR1) that 
\begin{align}LU({G}(U(C_1))){\bm{\odot}} %
      LU({G}(U(C_2))) \subseteq LU({G}(U(LU(C_1){\bm{\odot}} LU(C_2)))).%
      \label{in1}
      \end{align}
   By Definition \ref{def:lgets} and Lemma \ref{minmaxtMLUB}, inclusion (\ref{in1}) can be rewritten as  
   \begin{align}
{G}(U(C_1))\boxdot {G}(U(C_2))\subseteq LU(G(U(C_1\boxdot C_2)).\label{in2}
     \end{align}
     Hence 
      \begin{align}
{G}(U(C_1))\boxdot {G}(U(C_2))\leq U(G(U(C_1\boxdot C_2)),\label{in3}
     \end{align}
     which is equivalent to 
      \begin{align}
{G}(U(C_1))\boxdot {G}(U(C_2))\leq \Min U(G(U(C_1\boxdot C_2)).\label{in4}
     \end{align}

     Similarly, from Theorem \ref{fuzzthupcomplate} we obtain by (DR1) 
          \begin{align}
         \widehat{G}\big(\bigcap\{L(c\rightarrow & \,d)\mid c\in B_1, %
         d\in U(B_2)\} \big)\nonumber\\ 
         \subseteq& %
         \bigcap\{L(a\rightarrow b)\mid a\in \widehat{G}(LU(B_1)), %
         b\in U(\widehat{G}(LU(B_2)))\} \big). \label{in5}
              \end{align}
    Using Theorem \ref{thm:connect} and equation (2), inclusion (\ref{in5}) is equivalent with 
    \begin{align}
         LU\big({G}\big(U\big(&\bigcap\{L(c\rightarrow d)\mid c\in B_1, %
         d\in U(B_2)\}\big) \big)\big)\nonumber\\ 
         \subseteq& %
         \bigcap\{L(a\rightarrow b)\mid %
         a\in {G}(U(B_1)), %
         b\in U({G}(U(B_2)))\} \big). \label{in6}
              \end{align}
              Leveraging Definition~\ref{def:lgets} and Lemma~\ref{minmaxtMLUB}, we can express inclusion (\ref{in6}) equivalently as
    \begin{align}
         {G}\big(\Min U\big(\bigcap\{L(c\rightarrow d)\mid c\in B_1, %
         d\in U(B_2)\}\big) \big)
         \subseteq %
         L({G}(U(B_1)) \Rightarrow {G}(U(B_2))). \label{in7}
              \end{align}
    By applying Definition~\ref{def:lgets} again, inclusion (\ref{in7}) is shown to be equivalent to
     \begin{align}
         {G}(B_1\Rightarrow B_2)
         \leq %
         {G}(U(B_1)) \Rightarrow {G}(U(B_2)). \label{in8}
              \end{align}

(DT2): 
From Theorem  \ref{fuzzthupcomplate} we obtain by (DR2) that               
\begin{align}
\widehat{P}(LU(D_1)\bm{\odot} LU(D_2)) \subseteq \widehat{P}(LU(D_1))\bm{\odot} \widehat{P}(LU(D_2)).\label{in9}
 \end{align}
  Using repeatedly Theorem \ref{thm:connect}, equation (1)  and 
  Corollary \ref{cor:connect}, inclusion (\ref{in9}) is equivalent with 
  \begin{align}
\widehat{P}(LU(\{d_1 \odot d_2\mid d_1\in D_1, d_2\in D_2\})) \subseteq L({P}(D_1))\bm{\odot} L({P}(D_2)),\label{in10}
 \end{align}
 with 
  \begin{align}
L({P}(\Max LU(\{d_1 \odot d_2\mid d_1\in D_1, d_2\in D_2\})) \subseteq LU(L({P}(D_1)){\boxdot} L({P}(D_2))),\label{in11}
 \end{align}
  and with 
  \begin{align}
L({P}(D_1 \boxdot D_2)) \leq U(L({P}(D_1)){\boxdot} L({P}(D_2)))\label{in12}
 \end{align}

 Parts (DT3) - (DT4) can be proven similarly.
\end{proof}

\subsection*{Acknowledgements}
A. Ledda was  supported by PLEXUS, (Grant Agreement no. 101086295) a Marie Sklodowska-Curie action funded by the EU under the Horizon Europe Research and Innovation Programme.\\
The same author expresses his gratitude for the support of the Italian Ministry of Education, University and Research through the PRIN 2022 project ``Developing Kleene Logics and their Applications'' (DeKLA), project code: 2022SM4XC8.

\end{document}